\documentclass[a4paper,leqno,draft]{article}
\usepackage{amssymb,amsmath,amsfonts,amsthm,
mathrsfs}
\newtheorem{theorem}{Theorem}[section]
\newtheorem{corollary}[theorem]{Corollary}

\newtheorem{proposition}[theorem]{Proposition}
\newtheorem{definition}[theorem]{Definition}
\theoremstyle{definition}
\newtheorem{remark}[theorem]{Remark}
\newtheorem{example}[theorem]{Example}

\newcommand{\wt}[1]{\widetilde{#1}}

\newcommand{\Cinf}{\ensuremath{\mathcal{C}^\infty}}
\newcommand{\Cinfc}{\ensuremath{\mathcal{C}^\infty_{\text{c}}}}
\newcommand{\D}{\ensuremath{{\cal D}}}
\renewcommand{\S}{\mathscr{S}}
\newcommand{\E}{\ensuremath{{\cal E}}}
\newcommand{\OM}{\ensuremath{{\cal O}_\mathrm{M}}}
\newcommand{\LL}{\mathcal{L}}

\newcommand{\mb}[1]{\ensuremath{\mathbb{#1}}}
\newcommand{\N}{\mb{N}}

\newcommand{\R}{\mb{R}}
\newcommand{\C}{\mb{C}}

\newcommand{\G}{\ensuremath{{\cal G}}}
\newcommand{\Gt}{\ensuremath{{\cal G}_\tau}}

\newcommand{\Gc}{\ensuremath{{\cal G}_\mathrm{c}}}
\newcommand{\Gcinf}{\ensuremath{{\cal G}^\infty_\mathrm{c}}}
\newcommand{\GS}{\G_{{\, }\atop{\hskip-4pt\scriptstyle\S}}\!}
\newcommand{\EM}{\ensuremath{{\cal E}_{M}}}

\newcommand{\Et}{\ensuremath{{\cal E}_{\tau}}}

\newcommand{\Nt}{\ensuremath{{\cal N}_{\tau}}}
\newcommand{\Neg}{\mathcal{N}}

\newcommand{\Ginf}{\ensuremath{\G^\infty}}
\newcommand{\Gtinf}{\mathcal{G}^{\infty}_\tau}
\newcommand{\GSinf}{\G^\infty_{{\, }\atop{\hskip-3pt\scriptstyle\S}}}

\newcommand{\lara}[1]{\langle #1 \rangle}
\newcommand{\WF}{\mathrm{WF}}

\newcommand{\supp}{\operatorname{supp}}
\newcommand{\singsupp}{\operatorname{sing\,supp\hspace{0.5pt}}}
\newcommand{\Char}{\ensuremath{\text{Char}}}
\newcommand{\zs}{\setminus 0}
\newcommand{\CO}[1]{\ensuremath{T^*(#1) \zs}}
\newcommand{\ssc}{\mathrm{sc}}
\newcommand{\musupp}{\mu\mathrm{supp}}

\newfont{\bigmath}{cmr12 at 13pt}

\newfont{\grecomath}{cmmi12 at 15pt}

\newcommand{\esp}{\mathrm{e}}

\newcommand{\beq}{\begin{equation}}
\newcommand{\eeq}{\end{equation}}

\newcommand{\eps}{\varepsilon}

\newcommand{\Om}{\Omega}

\newcommand{\mM}{\mathcal{M}}
\newcommand{\mF}{\mathcal{F}}

\newcommand{\MPhi}{\mathcal{M}_\Phi}

\newcommand{\rfunc}{{\rm{basic}}}

\newcommand{\dslash}{d\hspace{-0.4em}{ }^-\hspace{-0.2em}}

\begin{document}

\title{{\bf Generalised Fourier integral operator methods for hyperbolic equations with singularities}}
\author{Claudia Garetto\footnote{Supported by JRF, Imperial College London}\\[0.1cm]
Department of Mathematics\\
Imperial College London\\
\texttt{c.garetto@imperial.ac.uk}\\
\  \\
Michael Oberguggenberger\footnote{Partially supported by FWF (Austria), grant Y237}\\[0.1cm]
Institut f\"ur Grundlagen der Bauingenieurwissenschaften\\
Leopold-Franzens-Universit\"at Innsbruck\\
\texttt{michael.oberguggenberger@uibk.ac.at}
}
\maketitle
\begin{abstract}
This article addresses linear hyperbolic partial differential equations and pseudodifferential equations with strongly singular coefficients and data, modelled as members of algebras of generalised functions. We employ the recently developed theory of generalised Fourier integral operators to construct parametrices for the solutions and to describe propagation of singularities in this setting. As required tools, the construction of generalised solutions to eikonal and transport equations is given and results on the microlocal regularity of the kernels of generalised Fourier integral operators are obtained.\\

AMS 2000 MSCS: 35S30, 46F30; 35B65.
\end{abstract}

\setcounter{section}{-1}
\section{Introduction}

This article is part of a series of papers aiming at developing microlocal analysis in Colombeau algebras of generalised functions \cite{Garetto:06a, Garetto:ISAAC07, GGO:03, GH:05, GHO:09, HOP:06, NedPilSca:98, O:07}. While these algebras originally have been introduced for handling nonlinear partial differential equations with distributional data, they have turned out to be equally useful for treating linear hyperbolic equations with strongly singular coefficients and data in the past decade. Consider the Cauchy problem for the linear hyperbolic system
\beq
\label{Cauchy_1}
\begin{split}
\partial_t u &=\sum_{j=1}^m A_j(x,t)\partial_j u+B(x,t)u+F(x,t),\qquad (x,t)\in\R^{m+1},\\
u(x,0)&=g(x),\qquad\qquad\qquad\qquad\qquad\qquad\qquad\qquad x\in\R^m.
\end{split}
\eeq
In case the initial data $g$ or the driving term $F$ are distributions and the coefficient matrices are, e.g., discontinuous, the system formally contains multiplicative products of distributions. Simple examples show \cite{Hurd:68} that such systems may fail to have solutions in the sense of distributions (referred to as classical solutions in the sequel). For second order strictly hyperbolic equations, the minimal regularity of the principal coefficients is log-Lipschitz regularity in time for admitting unique classical solutions \cite{Colombini:79, Colombini:95}. For transport equations, one may go down to BV-coefficients and measure data (see the recent surveys \cite{Ambrosio:08, Haller:08}). Differentiating the equation and trying to establish an equation for the higher derivatives of the formal solution already brings one beyond the admissible non-regularity, as will, e.g., letting $B$ be a delta potential or $F$ a space-time white noise.

The Colombeau algebra $\G(\R^{m+1})$ is a differential algebra containing the space of distributions ${\mathcal D}'(\R^{m+1})$ as a subspace and thus it provides a framework in which all operations arising in (\ref{Cauchy_1}) are meaningful. Indeed, existence and uniqueness of solutions to problem (\ref{Cauchy_1}) in the Colombeau algebra (as well as its dual) have been proven under various conditions: $A_j$ and $B$ real valued $n\times n$ matrices with entries in $\G(\R^{m+1})$, $A_j$ symmetric \cite{LO:91}; symmetric hyperbolic systems of pseudodifferential operators with Colombeau symbols \cite{GH:03, O:07}, strictly hyperbolic systems of pseudodifferential operators \cite{GO:10b}.

In the classical setting, Fourier integral operators (FIOs) and microlocal methods are used to study the dependence of the solution on the initial data, in particular, propagation of singularities. Indeed, it is the purpose of this paper to derive a generalised FIO-representation of the Colombeau solution and to predict its (generalised) wave front set.
In the Colombeau framework, pseudodifferential operators and their application to microlocal regularity and hypoellipticity have been developed in \cite{Garetto:06a, GGO:03, GH:05, GH:03, HOP:06, NedPilSca:98}, Fourier integral operators with generalised phase and amplitude have been developed in \cite{Garetto:ISAAC07, GHO:09}. The present paper initiates the application of generalised FIOs to hyperbolic equations with Colombeau coefficients. We shall show how to construct FIO-parametrices for transport equations and scalar hyperbolic pseudodifferential equations. Further, the generalised wave front sets of the kernels of generalised FIOs are determined, as well as the propagation of the generalised wave front sets through application of a generalised FIO. Finally, we study the generalised Hamlitonian flow and prove various results on propagation of singularities.

In short, it is the purpose of this paper to solve generalised strictly hyperbolic problems, which might be generated by singular coefficients and data, by means of FIO-techniques placed in the Colombeau context. Generalised pseudodifferential and Fourier integral operators as well as the microlocal tools in the Colombeau framework that have been developed over the past years, provide new and powerful tools and techniques which were not available at the time of the first work on hyperbolic generalised systems and equations in \cite{GH:03, LO:91, O:89, O:92}. On the other hand, the application of FIOs to generalised strictly hyperbolic problems is very natural also in the Colombeau context and will have interesting connections with systems with multiple characteristics and singularities. This and the singularity structure in the case of systems and of higher order equations is the subject of ongoing research and will be published elsewhere.

We comment on some technical aspects of the theory. Elements of the Colombeau algebra $\G(\R^{m+1})$ are equivalence classes of nets $(u_\eps)_\eps$ of ${\mathcal C}^\infty$-functions satisfying asymptotic bounds of order $\eps^{-N}$ in terms of the local $L^\infty$-seminorms of all derivatives as $\eps\to 0$. The generalised wave front set is also defined by means of decay in the dual variable and uniform asymptotic bounds as $\eps\to 0$. The existence theory is based on energy estimates and Gronwall's inequality and requires more restrictive asymptotic bounds of the type $|\log\varepsilon|$ at certain places; propagation of regularity may require additional so-called slow scale estimates. The necessity of such stronger bounds has been argued in \cite{GGO:03, HO:03, O:92}. When the coefficients are discontinuous functions or distributions, corresponding members of the Colombeau algebra satisfying the stronger bounds can always be constructed by suitable regularisation \cite{O:89}. Admitting distributional coefficients in hyperbolic equations may lead to infinite propagation speed \cite{LO:91}. To avoid this phenomenon, one may assume that the coefficients are constant for large $|x|$, as is often done in the classical case as well. Matters are even more simplified by assuming that the coefficients are compactly supported in $x$. This still allows one to model and study singularities as strong as desired (in the non-constant regime), but facilitates the application of Sobolev estimates.

It may also be useful to employ Colombeau spaces with asymptotic estimates based on Sobolev spaces. It has been shown in \cite{GH:03} that the first order scalar hyperbolic pseudodifferential problem
\beq
\label{Cauchy_2}
\partial_t u-i\,a(t,x,D_x)u=f,\qquad\qquad u(0)=g,
\eeq
with a real valued principal part has a unique solution in the Colombeau algebra ${\mathcal G}_{2,2}([0,T]\times\R^n)$ based on $H^\infty$. Assuming that the symbol $a$ is compactly supported in $x$, the Sobolev embedding theorem yields existence and uniqueness of a Colombeau solution in $\G([0,T]\times\R^n)$ as well.

Proofs of the mentioned existence and uniqueness results are based on a priori $L^2$-estimates in the representing nets of smooth functions. These methods have been successful regarding the well-posedness of the Cauchy problems under consideration but do not provide deep qualitative information on the solution. This already applies to the classical case of hyperbolic systems with $\Cinf$-coefficients and has been one of the main motivations for the constructive FIO approach introduced by Duistermaat and H\"ormander \cite{Dui:96, DH:72, Hoer:71}, which we extend here to the Colombeau framework. For a survey on local and global regularity of Fourier integral operators on $L^p$ spaces as well as on Colombeau spaces we refer the reader to \cite{Ruzhansky:09}.

The plan of the paper is as follows. In Section 1, we recall required notions from the Colombeau theory of generalised functions, in particular, microlocal tools, generalised symbols and phase functions and generalised FIOs. Section 2 is devoted to scalar first order hyperbolic partial differential equations with Colombeau coefficients. We show how the generalised solution can be represented by means of a generalised FIO. This includes the construction of Colombeau solutions to the eikonal equations and to the transport equation for the generalised symbol. In Section 3, we treat the case of first order hyperbolic pseudodifferential equations whose principal part depends on time only. This has the advantage that generalised solutions to the eikonal equation can be given explicitly (the more intricate $t$- and $x$-dependent case is postponed to work in preparation on hyperbolic problems with singularities). We solve an infinite system of transport equations yielding the asymptotic expansion of the generalised symbol, and construct an FIO parametrix of the generalised solution. Section 4 addresses microlocal properties of the kernel of a generalised FIO as well as its action on a Colombeau generalised functions. This allows one to compute the generalised wave front set (in space and time) of solutions to first order hyperbolic partial differential equations with non-smooth coefficients depending on time. Further, the spatial wave front set of solutions to first order partial differential equations at fixed time is computed, when the coefficients are generalised functions depending on space and time, in case the Hamiltonian flow has a limit. We show that the generalised wave front set is invariant under the Hamiltonian flow for time-dependent coefficients. The paper concludes with some explicit examples involving jump discontinuities and delta functions in the coefficients.

\section{Basic notions}
\subsection{Basic notions of Colombeau theory}
\label{sec_basic_co}
This section gives some background on Colombeau techniques used in the sequel of this paper. As main sources we refer to \cite{Garetto:05b, GGO:03, GH:05, GKOS:01}.

{\bf Nets of complex numbers.}
A net $(u_\eps)_\eps$ in $\C^{(0,1]}$ is said to be \emph{strictly nonzero} if there exist $r>0$ and $\eta\in(0,1]$ such that $|u_\eps|\ge \eps^r$ for all $\eps\in(0,\eta]$. For several regularity issues we will make use of the concept of \emph{slow scale net (s.s.n)}. A slow scale net is a net $(r_\eps)_\eps\in\C^{(0,1]}$ such that
\[
\forall q\ge 0\,\qquad\qquad\qquad\qquad |r_\eps|^q=O(\eps^{-1})\qquad \text{as $\eps\to 0$}.
\]
A net $(u_\eps)_\eps$ in $\C^{(0,1]}$ is said to be \emph{slow scale-strictly nonzero} if there exist a slow scale net $(s_\eps)_\eps$ and $\eta\in(0,1]$ such that $|u_\eps|\ge 1/s_\eps$ for all $\eps\in(0,\eta]$.

{\bf $\wt{\C}$-modules of generalised functions based on a locally convex topological vector space.}
The most common algebras of generalised functions of Colombeau type as well as the spaces of generalised symbols we deal with are introduced by referring to the following general models.

Let $E$ be a locally convex topological vector space topologised through the family of seminorms $\{p_i\}_{i\in I}$. The elements of
\[
\begin{split}
\mM_E &:= \{(u_\eps)_\eps\in E^{(0,1]}:\, \forall i\in I\,\, \exists N\in\N\quad p_i(u_\eps)=O(\eps^{-N})\, \text{as}\, \eps\to 0\},\\
\mM^\ssc_E &:=\{(u_\eps)_\eps\in E^{(0,1]}:\, \forall i\in I\,\, \exists (\omega_\eps)_\eps\, \text{s.s.n.}\quad p_i(u_\eps)=O(\omega_\eps)\, \text{as}\, \eps\to 0\},\\
\mM^\infty_E &:=\{(u_\eps)_\eps\in E^{(0,1]}:\, \exists N\in\N\,\, \forall i\in I\quad p_i(u_\eps)=O(\eps^{-N})\, \text{as}\, \eps\to 0\},\\
\Neg_E &:= \{(u_\eps)_\eps\in E^{(0,1]}:\, \forall i\in I\,\, \forall q\in\N\quad p_i(u_\eps)=O(\eps^{q})\, \text{as}\, \eps\to 0\},
\end{split}
\]

are called $E$-moderate, $E$-moderate of slow scale type, $E$-regular and $E$-negligible, respectively. We define the space of \emph{generalised functions based on $E$} as the factor space $\G_E := \mM_E / \Neg_E$.

The ring of \emph{complex generalised numbers}, denoted by $\wt{\C}$, is obtained by taking $E=\C$. Note that $\wt{\C}$ is not a field since by Theorem 1.2.38 in \cite{GKOS:01} only the elements which are strictly nonzero (i.e. the elements which have a representative strictly nonzero) are invertible and vice versa. Note that all the representatives of $u\in\wt{\C}$ are strictly nonzero once we know that there exists at least one which is strictly nonzero. When $u$ has a representative which is slow scale-strictly nonzero we say that it is \emph{slow scale-invertible}.

For any locally convex topological vector space $E$ the space $\G_E$ has the structure of a $\wt{\C}$-module. The ${\C}$-module $\G^\ssc_E:=\mM^\ssc_E/\Neg_E$ of \emph{slow scale regular generalised functions} and the $\wt{\C}$-module $\Ginf_E:=\mM^\infty_E/\Neg_E$ of \emph{regular generalised functions} are subspaces of $\G_E$ whose role is to describe different notions of regularity. We use the notation $u=[(u_\eps)_\eps]$ for the class $u$ of $(u_\eps)_\eps$ in $\G_E$. This is the usual way adopted in the paper to denote an equivalence class.

{\bf The most common Colombeau algebras.}
The Colombeau algebra $\G(\Om)=\EM(\Om)/\Neg(\Om)$ (see \cite{GKOS:01}) can be obtained as a ${\wt{\C}}$-module of $\G_E$-type by choosing $E=\Cinf(\Om)$. From a structural point of view $\Om\to\G(\Om)$ is a fine sheaf of differential algebras on $\R^n$. $\Gc(\Om)$ denotes the Colombeau algebra of generalised functions with compact support.

Regularity theory in the Colombeau context as initiated in \cite{O:92} is based on the subalgebra $\Ginf(\Om)$ of all elements $u$ of $\G(\Om)$ having a representative $(u_\eps)_\eps$ belonging to the set
\[
\EM^\infty(\Om):=\{(u_\eps)_\eps\in\E[\Om]:\ \forall K\Subset\Om\, \exists N\in\N\, \forall\alpha\in\N^n\quad \sup_{x\in K}|\partial^\alpha u_\eps(x)|=O(\eps^{-N})\, \text{as $\eps\to 0$}\}.
\]
The $\Ginf$-singular support of $u\in\G(\Om)$ ($\singsupp_{\Ginf}\, u$) is defined as the complement of the set of points $x$ such that $u|_V\in\Ginf(V)$ for some open neighborhood $V$ of $x$.

The intersection $\Ginf(\Om)\cap\Gc(\Om)$ will be denoted by $\Gcinf(\Om)$. In the course of the paper, for issues related to the Fourier transform, we will make use of the Colombeau algebra $\GS(\R^n)=\G_{\S(\R^n)}$ of generalised functions based on $\S(\R^n)$ and of the corresponding regular version $\GSinf(\R^n)$. Finally, we recall that the Colombeau algebra $\Gt(\R^n)$ of tempered generalised functions is defined as $\Et(\R^n)/\Nt(\R^n)$, where $\Et(\R^n)$ is the space
\[
\{(u_\eps)_\eps\in\OM(\R^n)^{(0,1]}:\ \forall\alpha\in\N^n\, \exists N\in\N\quad \sup_{x\in\R^n}(1+|x|)^{-N}|\partial^\alpha u_\eps(x)|=O(\eps^{-N})\ \text{as $\eps\to 0$}\}
\]
of $\tau$-moderate nets and $\Nt(\R^n)$ is the space
\[
\{(u_\eps)_\eps\in\OM(\R^n)^{(0,1]}:\ \forall\alpha\in\N^n\, \exists N\in\N\, \forall q\in\N\quad \sup_{x\in\R^n}(1+|x|)^{-N}|\partial^\alpha u_\eps(x)|=O(\eps^{q})\ \text{as $\eps\to 0$}\}
\]
of $\tau$-negligible nets. The subalgebra $\Gtinf(\R^n)$ of regular and tempered generalised functions is the quotient $\Et^{\infty}(\R^n)/\Nt(\R^n)$, where $\Et^{\infty}(\R^n)$ is the set of all $(u_\eps)_\eps\in\OM(\R^n)^{(0,1]}$ satisfying the following condition:
\[
\exists N\in\N\, \forall\alpha\in\N^n\, \exists M\in\N\qquad \sup_{x\in\R^n}(1+|x|)^{-M}|\partial^\alpha u_\eps(x)|=O(\eps^{-N}).
\]

{\bf Special types of Colombeau generalised functions in $\G(\R^n)$.}
An element of the Colombeau algebra is {\em real valued} if it is represented by a net of real valued smooth functions. We recall that $u\in\G(\R^n)$ is
\begin{itemize}
\item[-] \emph{of logarithmic type} if there exists a representative $(u_\eps)_\eps$ with the property that for all $K\Subset\R^n$
\[
\sup_{x\in K}|u_\eps(x)|=O(\log(\eps^{-1}))\qquad\quad \text{as $\eps\to 0$};
\]
\item[-] \emph{of slow scale logarithmic type} if there exists a representative $(u_\eps)_\eps$ with the property that for all $K\Subset\R^n$ there exist a slow scale net $(\mu_\eps)_\eps$ such that
\[
\sup_{x\in K}|u_\eps(x)|=O(\log(\mu_\eps))\qquad\quad \text{as $\eps\to 0$};
\]
\item[-] \emph{slow scale regular} if there exists a representative $(u_\eps)_\eps$ with the property that for all $K\Subset\R^n$, for all $\alpha\in\N^n$ there exist a slow scale net $(\mu_\eps)_\eps$ such that
\[
\sup_{x\in K}|\partial^\alpha u_\eps(x)|=O(\mu_\eps)\qquad\quad \text{as $\eps\to 0$}.
\]
\end{itemize}
Clearly the previous properties hold for all representatives of $u$ once they are known to be valid for one.

{\bf Topological dual of a Colombeau algebra.}
A topological theory of Colombeau algebras has been developed in \cite{Garetto:05a, Garetto:05b}. The duality theory for $\wt{\C}$-modules presented in \cite{Garetto:05b} in the framework of topological and locally convex topological $\wt{\C}$-modules, has provided the theoretical tools for dealing with the topological duals of the Colombeau algebras $\Gc(\Om)$ and $\G(\Om)$. We recall that $\LL(\G(\Om),\wt{\C})$ and $\LL(\Gc(\Om),\wt{\C})$ denote the space of all $\wt{\C}$-linear and continuous functionals on $\G(\Om)$ and $\Gc(\Om)$, respectively. For the choice of topologies given in \cite{Garetto:05a} one has the following chains of continuous embeddings:
\beq
\label{chain_1}
\Ginf(\Om)\subseteq\G(\Om)\subseteq\LL(\Gc(\Om),\wt{\C}),
\eeq
\beq
\label{chain_2}
\Gcinf(\Om)\subseteq\Gc(\Om)\subseteq\LL(\G(\Om),\wt{\C}),
\eeq
\beq
\label{chain_3}
\LL(\G(\Om),\wt{\C})\subseteq\LL(\Gc(\Om),\wt{\C}).
\eeq
In \eqref{chain_1} and \eqref{chain_2} the inclusion in the dual is given via integration $\big(u\to\big( v\to\int_\Om u(x)v(x)dx\big)\big)$ (for definitions and properties of the integral of a Colombeau generalised functions see \cite{GKOS:01}) while the embedding in \eqref{chain_3} is determined by the inclusion $\Gc(\Om)\subseteq\G(\Om)$. Since $\Om\to\LL(\Gc(\Om),\wt{\C})$ is a sheaf we can define the \emph{support of a functional $T$} (denoted by $\supp\, T$). In analogy with distribution theory we have that $\LL(\G(\Om),\wt{\C})$ can be identified with the set of functionals in $\LL(\Gc(\Om),\wt{\C})$ having compact support.

For questions related to regularity theory and microlocal analysis particular attention is given to those functionals in $\LL(\Gc(\Om),\wt{\C})$ and $\LL(\G(\Om),\wt{\C})$ which have a ``basic'' structure. In detail, we say that $T\in\LL(\Gc(\Om),\wt{\C})$ is ${\rfunc}$ if there exists a net $(T_\eps)_\eps\in\D'(\Om)^{(0,1]}$ fulfilling the following condition: for all $K\Subset\Om$ there exist $j\in\N$, $c>0$, $N\in\N$ and $\eta\in(0,1]$ such that
\[
\forall f\in\D_K(\Om)\, \forall\eps\in(0,\eta]\qquad\quad
|T_\eps(f)|\le c\eps^{-N}\sup_{x\in K,|\alpha|\le j}|\partial^\alpha f(x)|
\]
and $Tu=[(T_\eps u_\eps)_\eps]$ for all $u\in\Gc(\Om)$.\\
In the same way a functional $T\in\LL(\G(\Om),\wt{\C})$ is said to be $\rfunc$ if there exists a net  $(T_\eps)_\eps\in\E'(\Om)^{(0,1]}$ such that there exist $K\Subset\Om$, $j\in\N$, $c>0$, $N\in\N$ and $\eta\in(0,1]$ with the property
\[
\forall f\in\Cinf(\Om)\, \forall\eps\in(0,\eta]\qquad\quad
|T_\eps(f)|\le c\eps^{-N}\sup_{x\in K,|\alpha|\le j}|\partial^\alpha f(x)|
\]
and $Tu=[(T_\eps u_\eps)_\eps]$ for all $u\in\G(\Om)$. The sets of basic functionals on $\Gc(\Om)$ and $\G(\Om)$ are denoted by $\LL_{\rm{b}}(\Gc(\Om),\wt{\C})$ and $\LL_{\rm{b}}(\G(\Om),\wt{\C})$, respectively.

{\bf The Fourier transform on $\GS(\R^n)$, $\LL(\GS(\R^n),\wt{\C})$ and $\LL(\G(\Om),\wt{\C})$.}
The Fourier transform on $\GS(\R^n)$ is defined by the corresponding transformation at the level of representatives, as follows:
\[
\mF:\GS(\R^n)\to\GS(\R^n):u\to [(\widehat{u_\eps})_\eps].
\]
$\mF$ is a $\wt{\C}$-linear continuous map from $\GS(\R^n)$ into itself which extends to the dual in a natural way. In detail, we define the Fourier transform of $T\in\LL(\GS(\R^n),\wt{\C})$ as the functional in $\LL(\GS(\R^n),\wt{\C})$ given by
\[
\mF(T)(u)=T(\mF u).
\]
As shown in \cite[Remark 1.5]{Garetto:06a} $\LL(\G(\Om),\wt{\C})$ is embedded in $\LL(\GS(\R^n),\wt{\C})$ by means of the map
\[
\LL(\G(\Om),\wt{\C})\to\LL(\GS(\R^n),\wt{\C}):T\to \big(u\to T(({u_\eps}{\vert_\Om})_\eps+\Neg(\Om))\big).
\]
In particular, when $T$ is a basic functional in $\LL(\G(\Om),\wt{\C})$ we have from \cite[Proposition 1.6, Remark 1.7]{Garetto:06a} that the Fourier transform of $T$ is the tempered generalised function obtained as the action of $T(y)$ on $\esp^{-iy\xi}$, i.e., $\mF(T)=T(\esp^{-i\cdot\xi})=(T_\eps(\esp^{-i\cdot\xi}))_\eps+\Nt(\R^n)$. More precisely, $\mF(T)$ belongs to $\Ginf_\tau(\R^n)$.

{\bf Microlocal analysis in the Colombeau context: the $\Ginf$-wave front set for generalised functions and functionals.}
For an introduction to microlocal analysis in the Colombeau context we refer to \cite{Garetto:06a, GH:05}. Here we only  recall those microlocal concepts which we will employ in the final section of the paper. The $\Ginf$-wave front set of $u\in\Gc(\Om)$ ($\WF_\Ginf u$) is defined as the complement of the set of points $(x_0,\xi_0)\in\CO{\Om}$ fulfilling the following property: there exists a representative $(u_\eps)_\eps$ of $u$, a cut-off function $\varphi\in\mathcal{C}^\infty_{\rm{c}}(\Om)$ with $\varphi(x_0)=1$, a conic neighborhood $\Gamma$ of $\xi_0$ and a number $N$ such that for all $l\in\mathbb{R}$
\[
\sup_{\xi\in\Gamma}\,\lara{\xi}^l |\widehat{\varphi u_\eps}(\xi)| = O(\eps^{-N})\qquad \text{as}\ \eps\to 0.
\]
By construction $\pi_\Om \WF_{\Ginf} u=\singsupp_{\Ginf}u$. In addition, Theorem 3.11 in \cite{GH:05} shows that $\WF_{\Ginf}$ coincides with the set ${\rm{W}}_{{\rm{cl}},\Ginf}(u):=\bigcap_{Au\in\Ginf(\Om)}\Char(A)$, where the intersection is taken over all the standard properly supported pseudodifferential operators $A\in\Psi^0(\Om)$ such that $Au\in\Ginf(\Om)$. The adjective standard refers to symbols which do not depend on the parameter $\eps$ but belong to the usual H\"{o}rmander classes. This pseudodifferential characterisation provides the blueprint for extending the notion of wave front set from $\G(\Om)$ to the dual $\LL(\Gc(\Om),\wt{\C})$. In detail, for $T\in\LL(\Gc(\Om),\wt{\C})$ we define
\beq
\label{wf_dual}
\WF_{\Ginf}(T)={\rm{W}}_{{\rm{cl}},\Ginf}(T):=\bigcap_{AT\in\Ginf(\Om)}\Char(A)
\eeq
where the intersection is taken over all the standard properly supported pseudodifferential operators $A\in\Psi^0(\Om)$ such that $AT\in\Ginf(\Om)$. It follows that $\pi_{\Om}(\WF_{\Ginf}T)=\singsupp_{\Ginf} T$.

A useful characterisation of $\WF_{\Ginf}(T)$, in terms of estimates at the Fourier transform level, is valid when $T$ is basic. It involves the sets of generalised functions $\Ginf_{\S\hskip-2pt,0}(\Gamma)$, where $\Gamma$ is a conic subset of $\R^n\setminus 0$, of all tempered generalised functions $u$ having a representative $(u_\eps)_\eps$ fulfilling the following condition:
\[
\exists N\in\N\, \forall l\in\R\qquad\quad \sup_{\xi\in\Gamma}\lara{\xi}^l|u_\eps(\xi)|=O(\eps^{-N})\quad \text{as $\eps\to 0$}.
\]
Let $T\in\LL(\Gc(\Om),\wt{\C})$. Theorem 3.13 in \cite{Garetto:06a} shows that $(x_0,\xi_0)\not\in\WF_{\Ginf} T$ if and only if there exists a conic neighborhood $\Gamma$ of $\xi_0$ and a cut-off function $\varphi\in\Cinfc(\Om)$ with $\varphi(x_0)=1$ such that $\mF(\varphi T)\in\Ginf_{\S\hskip-2pt,0}(\Gamma)$.

\subsection{Generalised Fourier integral operators: generalised symbols and phase functions}
In this subsection we collect some basic notions concerning generalised pseudodifferential and Fourier integral operators. For a detailed presentation we refer to \cite{Garetto:ISAAC07, GHO:09}.

{\bf Generalised symbols.}
Let $\Om$ be an open subset of $\R^n$. By a \emph{generalised symbols} of order $m$ we mean an element of the space $\G_E$ based on $E=S^m_{\rho,\delta}(\Om\times\R^p)$. Analogously, $\G^\ssc_E$ with $E=S^m_{\rho,\delta}(\Om\times\R^p)$ is the space of slow scale regular generalised symbol. We say that a generalised symbol $a$ of order $m$ is \emph{regular} if it has a representative $(a_\eps)_\eps$ fulfilling the following condition:
\[
\forall K\Subset\Om\, \exists N\in\N\, \forall j\in\N\qquad\qquad |a_\eps|^{(m)}_{K,j}=\sup_{|\alpha|+|\beta|\le j} |a_\eps|^{(m)}_{K,\alpha,\beta}=O(\eps^{-N})\quad \text{as $\eps\to 0$}.
\]
A notion of asymptotic expansion for generalised symbols has been introduced in \cite[Subsection 2.5]{Garetto:ISAAC07} and employed in developing a complete symbolic calculus for generalised pseudodifferential operators, see \cite{Garetto:ISAAC07, GGO:03}. Finally, we recall that the \emph{conic support} of a generalised symbol $a$ of order $m$ is the complement of the set of points $(x_0,\xi_0)\in\Om\times\R^p$ such that there exists a relatively compact open neighborhood $U$ of $x_0$, a conic open neighborhood $\Gamma$ of $\xi_0$ and a representative $(a_\eps)_\eps$ of $a$ satisfying the condition
\begin{equation}
\label{cond_conic_supp}
\forall\alpha\in\N^p\, \forall\beta\in\N^n\, \forall q\in\N\quad \sup_{x\in U,\xi\in\Gamma}\lara{\xi}^{-m+\rho|\alpha|-\delta|\beta|}|\partial^\alpha_\xi\partial^\beta_x a_\eps(x,\xi)|=O(\eps^q)\quad \text{as $\eps\to 0$}.
\end{equation}
By definition, ${\rm{cone\, supp}}\, a$ is a closed conic subset of $\Om\times\R^p$.

{\bf Generalised phase functions.}
A \emph{phase function} $\phi(y,\xi)$ on $\Om\times\R^p$ is a smooth function on $\Om\times\R^p\setminus 0$, real valued, positively homogeneous of degree $1$ in $\xi$ such that $\nabla_{y,\xi}\phi(y,\xi)\neq 0$ for all $y\in\Om$ and $\xi\in\R^p\setminus 0$. We denote the set of all phase functions on $\Om\times\R^p$ by $\Phi(\Om\times\R^p)$ and the set of all nets in $\Phi(\Om\times\R^p)^{(0,1]}$ by $\Phi[\Om\times\R^p]$. We recall that $S^1_{\rm{hg}}(\Om\times\R^p\setminus 0)$ is the space of symbols on $\Om\times\R^p\setminus 0$ homogeneous of order $1$ in $\xi$, i.e. $ \sup_{x\in K,\xi\in\R^p\setminus 0}|\xi|^{-1+\alpha}|\partial^\alpha_\xi\partial^\beta_x a(x,\xi)|<\infty$ for all $K\Subset\Om$, for all $\alpha\in\N^p$ and $\beta\in\N^n$.
\begin{definition}
\label{def_phase_moderate}
An element of $\mathcal{M}_\Phi(\Om\times\R^p)$ is a net $(\phi_\eps)_\eps\in\Phi[\Om\times\R^p]$ satisfying the conditions:
\begin{itemize}
\item[(i)] $(\phi_\eps)_\eps\in\mathcal{M}_{S^1_{\rm{hg}}(\Om\times\R^p\setminus 0)}$,
\item[(ii)] for all $K\Subset\Om$ the net $$\biggl(\inf_{y\in K,\xi\in\R^p\setminus 0}\biggl|\nabla_{y,\xi} \phi_\eps\biggl(y,\frac{\xi}{|\xi|}\biggr)\biggr|^2\biggr)_\eps$$ is strictly nonzero.
\end{itemize}
On $\MPhi(\Om\times\R^p)$ we introduce the equivalence relation $\sim$ as follows: $(\phi_\eps)_\eps\sim(\omega_\eps)_\eps$ if and only if $(\phi_\eps-\omega_\eps)\in\Neg_{S^1_{\rm{hg}}(\Om\times\R^p\setminus 0)}$. The elements of the factor space  $$\wt{\Phi}(\Om\times\R^p):={\mathcal{M}_\Phi(\Om\times\R^p)}/{\sim}.$$
will be called \emph{generalised phase functions}.

Finally, we say that $\phi\in\wt{\Phi}(\Om\times\R^p)$ is a \emph{slow scale generalised phase function} if it has a representative $(\phi_\eps)_\eps\in\mathcal{M}^\ssc_{S^1_{\rm{hg}}(\Om\times\R^p\setminus 0)}$ such that the net in $(ii)$ is slow scale strictly nonzero.
\end{definition}
In the sequel $\Om'$ is an open subset of $\R^{n'}$. We denote by $\Phi[\Om';\Om\times\R^p]$ the set of all nets $(\phi_\eps)_{\eps\in(0,1]}$ of continuous functions on $\Om'\times\Om\times\R^p$ which are smooth on $\Om'\times\Om\times\R^p\setminus\{0\}$ and such that $(\phi_\eps(x,\cdot,\cdot))_\eps\in\Phi[\Om\times\R^p]$ for all $x\in\Om'$.
\begin{definition}
\label{def_phase_x_moderate}
An element of $\mM_{\Phi}(\Om';\Om\times\R^p)$ is a net $(\phi_\eps)_\eps\in\Phi[\Om';\Om\times\R^p]$ satisfying the conditions:
\begin{itemize}
\item[(i)] $(\phi_\eps)_\eps\in\mM_{S^1_{\rm{hg}}(\Om'\times\Om\times\R^p\setminus 0)}$,
\item[(ii)] for all $K'\Subset\Om'$ and $K\Subset\Om$ the net
\beq
\label{net_FIO}
\biggl(\inf_{x\in K',y\in K,\xi\in\R^p\setminus 0}\biggl|\nabla_{y,\xi} \phi_\eps\biggl(x,y,\frac{\xi}{|\xi|}\biggr)\biggr|^2\biggr)_\eps
\eeq
is strictly nonzero.
\end{itemize}
On $\mM_{\Phi}(\Om';\Om\times\R^p)$ we introduce the equivalence relation $\sim$ as follows: $(\phi_\eps)_\eps\sim(\omega_\eps)_\eps$ if and only if $(\phi_\eps-\omega_\eps)_\eps\in\Neg_{S^1_{\rm{hg}}(\Om'\times\Om\times\R^p\setminus 0)}$. The elements of the factor space
\[
\wt{\Phi}(\Om';\Om\times\R^p):=\mM_{\Phi}(\Om';\Om\times\R^p) / \sim.
\]
are called \emph{generalised phase functions with respect to the variables in $\Om\times\R^p$}.  If $\phi\in\wt{\Phi}(\Om';\Om\times\R^p)$ has a representative $(\phi_\eps)_\eps\in\mM^\ssc_{S^1_{\rm{hg}}(\Om'\times\Om\times\R^p\setminus 0)}$ such that the net in \eqref{net_FIO} is slow scale strictly nonzero then it is called \emph{slow scale generalised phase function with respect to the variables in $\Om\times\R^p$}.
\end{definition}

{\bf Slow scale critical points.}
\begin{definition}
\label{def_C_phi_ssc}
Let $\phi\in\wt{\Phi}(\Om\times\R^p)$. We define $C^\ssc_{\phi}\subseteq\Om\times\R^p\setminus 0$ as the complement of the set of all $(x_0,\xi_0)\in\Om\times\R^p\setminus 0$ with the property that there exist a relatively compact open neighborhood $U(x_0)$ of $x_0$ and a conic open neighborhood $\Gamma(\xi_0)\subseteq\R^p\setminus 0$ of $\xi_0$ such that the generalised function $|\nabla_\xi\phi(\cdot,\cdot)|^2$ is slow scale-invertible on $U(x_0)\times\Gamma(\xi_0)$. We set $\pi_\Om(C^\ssc_{\phi})=S^\ssc_{\phi}$ and $R^\ssc_{\phi}=(S^\ssc_{\phi})^{{\rm{c}}}$.
\end{definition}
By construction $C^\ssc_{\phi}$ is a conic closed subset of $\Om\times\R^p\setminus 0$ and $R^\ssc_{\phi}\subseteq R_{\phi}\subseteq\Om$ is open. It is routine to check that the region $C^\ssc_\phi$ coincides with the classical one $\{(x,\xi)\in\Om\times\R^p\setminus 0:\, \nabla_\xi\phi(x,\xi)=0\}$ when $\phi$ is a standard phase function independent of $\eps$.

{\bf Generalised Fourier integral operators.}
Let $\phi\in\wt{\Phi}(\Om';\Om\times\R^p)$, $a\in\G_{S^m_{\rho,\delta}(\Om'\times\Om\times\R^p})$ and $u\in\Gc(\Om)$. The generalised oscillatory integral
\[
I_{\phi}(a)(u)(x)=\int_{\Om\times\R^p}\esp^{i\phi(x,y,\xi)}a(x,y,\xi)u(y)\, dy\,\dslash\xi
\]
defines a generalised function in $\G(\Om')$ and the map
\begin{equation}
\label{def_A_fourier}
A:\Gc(\Om)\to\G(\Om'):u\to I_{\phi}(a)(u)
\end{equation}
is continuous. The operator $A$ defined in \eqref{def_A_fourier} is called \emph{generalised Fourier integral operator} with amplitude $a$ and phase function $\phi\in\wt{\Phi}(\Om';\Om\times\R^p)$. From Theorem 4.6 in \cite{Garetto:ISAAC07} we have that when phase function and amplitude are both slow scale regular the corresponding generalised Fourier integral operator maps $\Gcinf(\Om)$ continuously into $\Ginf(\Om')$. If $a\in\G^\ssc_{S^{-\infty}(\Om'\times\Om\times\R^p)}$ then $A$ maps $\Gc(\Om)$ into $\Ginf(\Om')$. Pseudodifferential operators are special Fourier integral operators with $\phi(x,y,\xi)=(x-y)\xi$ in \eqref{def_A_fourier}.

{\bf Composition of a generalised Fourier integral operator with a generalised pseudodifferential operator.} We conclude this survey on generalised Fourier integral operators by studying the composition with a generalised pseudodifferential operator. First of all we focus on Fourier integral operators of the form
\[
F_\omega(b)(u)(x)=\int_{\R^n}\esp^{i\omega(x,\eta)}b(x,\eta)\widehat{u}(\eta)\, \dslash\eta,
\]
where $\omega\in\G_{S^1_{\rm{hg}}(\Om'\times\R^n\setminus 0)}$, $b\in\G_{S^m(\Om'\times\R^n)}$ and $u\in\Gc(\Om)$. Note that $\phi(x,y,\eta):=\omega(x,\eta)-y\eta$ is a well-defined generalised phase function belonging to $\wt{\Phi}(\Om';\Om\times\R^n)$ and $I_\phi(b)=F_\omega(b)$. Theorem 5.11 in \cite{Garetto:ISAAC07} provides the composition formula stated below.

{\bf Composition formula for $a(x,D)F_\omega(b)$.}
Let $\omega\in\G^{\ssc}_{S^1_{\rm{hg}}(\Om\times\R^n\setminus 0)}$ have a representative $(\omega_\eps)_\eps$ such that $\nabla_x\omega_\eps\neq 0$ for all $\eps\in(0,1]$ and for all $K\Subset\Om$
\[
\biggl(\inf_{x\in K,\eta\in\R^n\setminus 0}\biggl|\nabla_{x}\omega_\eps\big(x,\frac{\eta}{|\eta|}\big)\biggr|\biggr)_\eps
\]
is slow scale strictly non-zero. Let $a\in\G^\ssc_{S^m(\Om\times\R^n)}$ and $b\in\G^\ssc_{S^l(\Om\times\R^n\setminus 0)}$ with $\supp_x\, b\Subset\Om$. Then the operator $a(x,D)F_\omega(b)$ has the following properties:
\begin{itemize}
\item[(i)] it maps $\Gcinf(\Om)$ into $\Ginf(\Om)$;
\item[(ii)] it is of the form
\[
\int_{\R^n}\esp^{i\omega(x,\eta)}h(x,\eta)\widehat{u}(\eta)\, \dslash\eta +r(x,D)u,
\]
where $h\in\G^\ssc_{S^{l+m}(\Om\times\R^n\setminus 0)}$ has asymptotic expansion given by the symbols
\[
h_\alpha(x,\eta)=\frac{\partial^\alpha_\xi a(x,\nabla_x\omega(x,\eta))}{\alpha!}D^\alpha_z\big(\esp^{i\overline{\omega}(z,x,\eta)}b(z,\eta)\big)|_{z=x},\qquad \alpha\in\N^n,
\]
with $\overline{\omega}(z,x,\eta):=\omega(z,\eta)-\omega(x,\eta)-\lara{\nabla_x\omega(x,\eta),z-x}$, and $r$ is slow scale regular and of order $-\infty$.
\end{itemize}

\section{Transport equations with generalised coefficients}
\label{sec_PDE}
In this section, we are concerned with the Cauchy problem for the first order hyperbolic equation
\beq
\begin{split}
\label{first_order_equ}
 D_t u &=\sum_{j=1}^n a_{1,j}(t,x)D_j u+a_0(t,x)u \\
 u(0,\cdot)&=u_0,
\end{split}
\eeq
where $D_j=D_{x_j}$, the coefficients $a_{1,j}$ are real valued Colombeau generalised functions in $\G(\R^{n+1})$, $a_0\in\G(\R^{n+1})$ and $u_0\in\Gc(\R^n)$. The following theorem combines the well-posedness results of \cite{LO:91} with the more recent investigations of \cite{GH:03}.
\begin{theorem}
\label{theo_first_order_equ}
Let the coefficients $a_{1,j}$, $j=1,...,n$ and $a_0$ be Colombeau generalised functions in $\G(\R^{n+1})$ compactly supported in $x$. Assume that the coefficients $a_{1,j}$ are real valued and $\partial_k a_{1,j}$ as well as $a_0$ are of logarithmic type (k,j=1,...,n).
Then:
\begin{itemize}
\item[(i)] For each $u_0\in\G(\R^n)$ the Cauchy problem \eqref{first_order_equ} has a unique solution $u\in\G(\R^{n+1})$.
\item[(ii)] If in addition $u_0\in\Gc(\R^n)$ then the solution $u$ is compactly supported in $x$.
\item[(iii)] If in addition $a_0$ as well as $u_0$ are real valued then the solution $u$ is a real valued generalised function.
\item[(iv)] If the coefficients $a_{1,j}$ and $a_0$ are slow scale regular and $\partial_k a_{1,j}$ (k,j=1,...,n) as well as $a_0$ are of slow scale logarithmic type then for each $u_0\in\Ginf(\R^n)$ the unique solution $u\in\G(\R^{n+1})$ of the Cauchy problem \eqref{first_order_equ} belongs to $\Ginf(\R^{n+1})$.
\item[(v)] Under the hypotheses of $(iv)$ if the initial data $u_0$ are slow scale regular then $u$ is slow scale regular as well.
\end{itemize}
\end{theorem}
The aim of this section is to prove that the solution $u$ of the Cauchy problem \eqref{first_order_equ} can be written as the action of a generalised Fourier integral operator $F_\phi(b)$ on the initial data $u_0\in\Gc(\R^n)$. This requires to determine the phase function $\phi$ and the symbol $b$.
\subsection{The generalised phase function and the characteristic curves}
\label{subsec_eikonal}
The generalised phase function $\phi$ is the solution of the eikonal equation determined by the principal part of the operator
\[
 D_t -\sum_{j=1}^n a_{1,j}(t,x)D_j-a_0(t,x)
\]
under the initial condition $\phi(0,x,\eta)=x\eta$. Thus one has to solve the linear Cauchy problem
\beq
\label{eikonal_0}
\begin{split}
\partial_t\phi(t,x,\eta)&=\sum_{j=1}^n a_{1,j}(t,x)\partial_j\phi(t,x,\eta),\\
\phi(0,x,\eta)&=x\eta.
\end{split}
\eeq
Under the hypotheses of Theorem \ref{theo_first_order_equ} on the coefficients $a_{1,j}$, we already know that there exists a unique solution $\phi\in\G(\R^{n+2})$. More precisely, it has the form
\[
\phi(t,x,\eta)=\sum_{h=1}^n \omega_h(t,x)\eta_h,
\]
where $\omega_h$, $h=1,...,n$, are solutions of the Cauchy problems
\beq
\label{eikonal_1}
\begin{split}
\partial_t\omega_h(t,x)&=\sum_{j=1}^n a_{j,1}(t,x)\partial_{j}\omega_h(t,x),\\
\omega_h(0,x)&=x_h.\\
\end{split}
\eeq
In the following proposition we describe the properties of the generalised functions $\omega_h$ more specifically. The solutions $\gamma_1,...,\gamma_n$ of the initial value problem
\beq
\label{charac_n}
\frac{d}{ds}\gamma_{h}(x,t,s)=-a_{1,h}(s,\gamma_{1}(x,t,s),\gamma_{2}(x,t,s),...,\gamma_{n}(x,t,s)),\qquad\quad \gamma_{h}(x,t,t)=x_j,\ h=1,...,n.
\eeq
are the components of the \emph{characteristic curve} $\gamma=(\gamma_1,...,\gamma_n)$ associated with the differential operator $\sum_{j=1}^n a_{1,j}(t,x)D_j$. Note that from Theorem 1.5.2 and Remark 1.5.3 in \cite{GKOS:01} this initial value problem is well-posed in $\G(\R^{n+2})$ when $a_{1,j}$ is compactly supported in $x$ and has first order $x$-derivatives of logarithmic type.

We make the following assumptions:
\begin{itemize}
\item[$(h1)$] the coefficients $a_{1,j}$ are real valued generalised functions in $\G(\R^{n+1})$, compactly supported with respect to $x$ with $\partial_k a_{1,j}$ of logarithmic type ($k,j=1,...,n$);
\item[$(h2)$] the coefficients $a_{1,j}$ are real valued slow scale regular generalised functions in $\G(\R^{n+1})$, compactly supported with respect to $x$ with $\partial_k a_{1,j}$ of slow scale logarithmic type (k,j=1,...,n).
\end{itemize}
The compact support property is introduced here to keep the presentation simple and could be relaxed.
\begin{proposition}
\label{prop_eikonal_PDE}
\leavevmode
\begin{itemize}
\item[(i)] Under the hypothesis $(h1)$ there exists a unique real valued solution $\omega_h(t,x)\in\G(\R^{n+1})$ of the Cauchy problem \eqref{eikonal_1}; $\omega_h(t,x)$ is the $h$-th component of the characteristic curve $\gamma(x,t,0)$.
\item[(ii)] Under the hypothesis $(h2)$ the solution $\omega_h$ is slow scale regular.
\end{itemize}
\end{proposition}
\begin{proof}
From Theorem \ref{theo_first_order_equ} $(i)$ and $(iii)$ it is clear that there exists a unique real valued Colombeau solution $\omega_h\in\G(\R^{n+1})$. Since the initial data $x_h$ are smooth and therefore slow scale regular we have from Theorem \ref{theo_first_order_equ} $(iv)$ and $(v)$ that $\omega_h$ is slow scale regular under the set of assumptions $(h2)$ on $a_{1,j}$ and $\partial_k a_{1,j}$, $k,j=1,...,n$. It remains to prove that $\omega_h(t,x)=\gamma_h(x,t,0)$. This comes from the fact that $\omega_h$ is constant along the characteristic curves $\gamma(x,t,s)$, i.e., working at the level of representatives,
\[
\frac{d}{ds}\omega_{h,\eps}(s,\gamma_{1,\eps}(x,t,s),\gamma_{2,\eps}(x,t,s),...,\gamma_{n,\eps}(x,t,s))=0.
\]
Hence,
\[
\omega_{h,\eps}(t,\gamma_{1,\eps}(x,t,t),...,\gamma_{n,\eps}(x,t,t))=\omega_{h,\eps}(0,\gamma_{1,\eps}(x,t,0),...,\gamma_{n,\eps}(x,t,0))
\]
for each $t\in\R$. This implies that $\omega_h(t,x)=\gamma_h(x,t,0)$ in $\G(\R^{n+1})$.
\end{proof}
Summarising we can state the following proposition.
\begin{proposition}
\label{prop_phase_PDE}
If $(h1)$ holds then the generalised phase function
\[
\phi(t,x,\eta)=\sum_{h=1}^n\omega_h(t,x)\eta_h=\sum_{h=1}^n\gamma_h(x,t,0)\eta_h
\]
solves the eikonal Cauchy problem \eqref{eikonal_0}.

If $(h2)$ holds then $\phi$ is a slow scale generalised phase function.
\end{proposition}
\begin{remark}
\label{rem_phase}
The Colombeau generalised function $\phi$ is actually a generalised symbol homogeneous of order 1 in $\eta$. With a certain abuse of language we employ the expression \emph{generalised phase function}, previously referred to $\phi'(t,x,y,\eta)=\phi(t,x,\eta)-y\eta$. Indeed, for $\phi'$ one has the typical invertibility condition on the gradient, i.e.,
\[
|\nabla\phi'(t,x,y,\eta)|=|(\nabla_t\phi(t,x,\eta),\nabla_x\phi(t,x,\eta),-\eta,\nabla_\eta\phi(t,x,\eta)-y)|\ge 1
\]
for all $\eta$ with $|\eta|=1$.
\end{remark}

\subsection{The transport equation for the generalised symbol}
\label{subsec_transport}
To compute the generalised symbol $b$ we need to solve the Cauchy problem \eqref{first_order_equ} with initial condition $1$ at $t=0$. In detail,
\beq
\begin{split}
\label{first_transport}
 D_t b &=\sum_{j=1}^n a_{1,j}(t,x)D_j b+a_0(t,x)b \\
 b(0,\cdot)&=1.
\end{split}
\eeq
We introduce the following set of hypotheses on $a_0$:
\begin{itemize}
\item[$(i1)$] $a_0$ is a generalised function in $\G(\R^{n+1})$, compactly supported in $x$ and of logarithmic type;
\item[$(i2)$] $a_0$ is a slow scale regular generalised function in $\G(\R^{n+1})$, compactly supported in $x$ with $0$-derivative of slow scale logarithmic type.
\end{itemize}
Using Theorem \ref{theo_first_order_equ} and integrating along the characteristics we have the following existence and uniqueness result.
\begin{theorem}
\label{theo_transport_first}
\leavevmode
\begin{itemize}
\item[(i)] Under the hypothesis $(h1)$ on the coefficients $a_{1,j}$ and the hypothesis $(i1)$ on $a_0$ there exists a unique solution $b\in \G(\R^{n+1})$ of the Cauchy problem \eqref{first_transport}.
\item[(ii)] Under the hypothesis $(h2)$ on the coefficients $a_{1,j}$ and the hypothesis $(i2)$ on $a_0$, the solution $b\in \G(\R^{n+1})$ of the Cauchy problem \eqref{first_transport} is slow scale regular.
\item[(iii)] Further, $b(t,x)=\esp^{i\beta(t,x)}$ with
\[
\beta(t,x)=\int_{0}^t a_0(s,\gamma_1(x,t,s),...,\gamma_n(x,t,s))\, ds.
\]
\end{itemize}
\end{theorem}
\subsection{Generalised FIO formula}
\label{subsec_FIO_formula}
A combination of Proposition \ref{prop_phase_PDE} with Theorem \ref{theo_transport_first} yields the following FIO formula.
\begin{proposition}
\label{prop_FIO_formula}
\leavevmode
\begin{itemize}
\item[(i)] Under the hypotheses $(h1)$ and $(i1)$ on the coefficients $a_{1,j}$ and $a_0$, respectively, the solution $u\in\G(\R^{n+1})$ of the Cauchy problem \eqref{first_order_equ} can be written as
\beq
\label{FIO_formula}
u(t,x)=F_\phi(b)(u_0)(t,x):=\int_{\R^n}\esp^{i\phi(t,x,\eta)}b(t,x)\widehat{u_0}(\eta)\, \dslash\eta,
\eeq
where the generalised phase function $\phi$ is defined by Proposition \ref{prop_phase_PDE} and $b\in \G(\R^{n+1})$ by Theorem \ref{theo_transport_first}.
\item[(ii)] Under the hypotheses $(h2)$ and $(i2)$ the formula \eqref{FIO_formula} holds with $\phi$ and $b$ slow scale regular.
\end{itemize}
\end{proposition}
\subsection{The non-homogeneous Cauchy problem}
We conclude this section by finding a solution formula for the non-homogeneous Cauchy problem
\beq
\begin{split}
\label{first_order_equ_nh}
 D_t u &=\sum_{j=1}^n a_{1,j}(t,x)D_j u+a_0(t,x)u+f(t,x) \\
 u(0,\cdot)&=u_0,
\end{split}
\eeq
where $f\in\G(\R^{n+1})$ is compactly supported with respect to $x$. Note that the Fourier integral operator $F_\phi(b)$ of Proposition \ref{prop_FIO_formula} is given by
\[
F_\phi(b)(u_0)(t,x)=b(t,x)u_0(\gamma(x,t,0)).
\]
It defines, for each $t\in\R$, a map
\[
U(t)=F_\phi(b)(t):\Gc(\R^n)\to \Gc(\R^n):u_0\to F_\phi(b)(u_0)(t,\cdot)
\]
such that $U(0)=I$ and
\[
U(t)^{-1}=\Gc(\R^n)\to\Gc(\R^n):v\to\frac{1}{b(t,\gamma(x,0,t))}v(\gamma(x,0,t)).
\]
\begin{theorem}
\label{theo_FIO_formula_nh}
\leavevmode
\begin{itemize}
\item[(i)] Under the hypotheses $(h1)$ and $(i1)$ on the coefficients $a_{1,j}$ and $a_0$, respectively, the solution $u\in\G(\R^{n+1})$ of the Cauchy problem \eqref{first_order_equ_nh} can be written as
\beq
\label{FIO_formula_nh}
u(t,x)=F_\phi(b)(t)\biggl(u_0+i\int_0^t\frac{1}{b(\tau,\gamma(\cdot,0,\tau))}f(\tau,\gamma(\cdot,0,\tau))d\tau\biggr)(x)
\eeq
where the generalised phase function $\phi$ is defined by Proposition \ref{prop_phase_PDE} and $b\in \G(\R^{n+1})$ by Theorem \ref{theo_transport_first}.
\item[(ii)] Under the hypotheses $(h2)$ and $(i2)$ the formula \eqref{FIO_formula_nh} holds with $\phi$ and $b$ slow scale regular.
\end{itemize}
\end{theorem}

\section{First order hyperbolic pseudodifferential equations with ge\-ne\-ra\-li\-zed symbols}
\label{sec_pseudo}
We now consider hyperbolic first order pseudodifferential equations of the type
\beq
\label{first_order_pseudo}
D_t u=a_1(t,D_x)u+a_0(t,x,D_x)u
\eeq
where $a_1$ and $a_0$ are generalised symbols of order $1$ and $0$, respectively, with $a_1$ real valued and independent of $x$. As mentioned in the Introduction,  we restrict ourselves to $t$-dependent principal parts. We begin by collecting what is known about equation \eqref{first_order_pseudo} in the Colombeau context. The following theorem is due to H\"ormann in \cite{GH:03}. The well-posedeness of the Cauchy problem
\beq
\label{CP_pseudo}
\begin{split}
 D_t u &=a_1(t,D_x)u+a_0(t,x,D_x)u+f \\
 u(0,\cdot)&=u_0,
\end{split}
\eeq
is intended in the Colombeau algebra $\G_{H^{\infty}((-T,T)\times\R^n)}$ based on $H^{\infty}((-T,T)\times\R^n)$. Here we use the notation $\G_{2,2}((-T,T)\times\R^n)$ introduced in \cite{BO:92}. The choice of this setting is motivated by a uniqueness issue: the solution $u$ to the problem \eqref{CP_pseudo} fails to be unique in the usual Colombeau algebra $\G([-T,T]\times\R^n)$ whereas is uniquely determined in $\G_{2,2}((-T,T)\times\R^n)$. Finally, with the expressions \emph{generalised symbol} and \emph{slow scale regular generalised symbol} we refer to the elements of the spaces $\G_{\Cinf([-T,T], S^m(\R^{2n}))}$ and $\G^\ssc_{\Cinf([-T,T], S^m(\R^{2n}))}$, respectively. Note that with the notation $S^m(\R^{2n})$ we intend symbols satisfying uniform estimates with respect to $x\in\R^n$ and $\xi\in\R^n$, i.e., $\sup_{(x,\xi)\in\R^{2n}}\lara{\xi}^{-m+|\alpha|}|\partial^\alpha_\xi\partial^\beta_x a(x,\xi)|<\infty$. In the sequel $(k,h)$ denotes the $\eta$-derivatives and $x$-derivatives of a symbol up to order $k$ and $h$, respectively.
\begin{theorem}
\label{theo_first_order_pseudo}
Let $a_1$ be a real valued generalised symbol of order $1$, $a_0$ a generalised symbol of order $0$, $f\in\G_{2,2}((-T,T)\times\R^n)$ and $u_0\in\G_{2,2}(\R^n)$.
\begin{itemize}
\item[(i)] If there exist representatives $(a_{1,\eps})_\eps$ and $(a_{0,\eps})_\eps$ of $a_1$ and $a_0$ respectively such that $(a_{1,\eps})_\eps$ is of log-type up to order $(k_n,l_n+1)$, with $k_n=3(\left\lfloor n/2\right\rfloor+1)$ and $l_n=2(n+2)$, and $(a_{0,\eps})_\eps$ is of log-type up to order $(k_n',j_n')$ with $k'_n=l'_n=\left\lfloor n/2\right\rfloor+1$, then the Cauchy problem \eqref{CP_pseudo} is well-posed in $\G_{2,2}((-T,T)\times\R^n)$.
\item[(ii)] If for large $|x|$ the net $(a_{0,\eps})_\eps$ does not depend on $t$ then one can set $k_n=1$, $l_n=n+2$, $k'_n=0$, $l'_n=n+1$ in $(i)$.
\item[(iii)] If $a_1$ and $a_0$ are slow scale regular generalised symbols and the log-type conditions in $(i)$ on $(a_{1,\eps})_\eps$ and $(a_{0,\eps})_\eps$ are replaced by slow scale log-type assumptions then for each $f\in\G^\infty_{2,2}((-T,T)\times\R^n)$ and $u_0\in\G^\infty_{2,2}(\R^n)$ the unique solution $u$ to the Cauchy problem \eqref{CP_pseudo} belongs to $\G^{\infty}_{2,2}((-T,T)\times\R^n)$.
\end{itemize}
\end{theorem}
Our aim is now to find an FIO formula by which to express the solution $u$. More precisely, we will construct a generalised FIO parametrix for the hyperbolic Cauchy problem
\beq
\label{CP_pseudo_hom}
\begin{split}
 D_t u &=a_1(t,D_x)u+a_0(t,x,D_x)u \\
 u(0,\cdot)&=u_0\in\Gc(\R^n).
\end{split}
\eeq
Under suitable moderateness assumptions we will get well-posedness and $\Ginf$-regularity.

\subsection{The generalised phase function and the eikonal equation}
\label{subsec_eikonal_pseudo}
We begin by determining the generalised phase function $\phi(t,x,\eta)$, i.e. by solving the following eikonal equation.
\begin{proposition}
\label{prop_phase_pseudo}
The eikonal equation
\beq
\label{eikonal_pseudo}
\begin{split}
\partial_t\phi(t,x,\eta)&=a_1(t,\nabla_x\phi(t,x,\eta)),\\
\phi(0,x,\eta)&=x\eta,
\end{split}
\eeq
has the solution
\[
\phi(t,x,\eta)=x\eta+\int_{0}^t a_1(s,\eta)\, ds
\]
\begin{itemize}
\item[(i)] in $\G_{\Cinf([-T,T], S^1(\R^{2n}))}$ if $a_1\in\G_{\Cinf([-T,T], S^1(\R^{2n}))}$,
\item[(ii)] in $\G^\ssc_{\Cinf([-T,T], S^1(\R^{2n}))}$ if $a_1\in\G^\ssc_{\Cinf([-T,T], S^1(\R^{2n}))}$.
\end{itemize}
\end{proposition}
\begin{remark}
(i) Note that $\phi(t,x,\eta)$ is a symbol of order $1$ with $\nabla_x\phi(t,x,\eta)=\eta$ but non-homogeneous with respect to $\eta$. This is not an obstacle in defining the Fourier integral operator
\[
\int_{\R^{n}}\esp^{i\phi(t,x,\eta)}b(t,x,\eta)\widehat{v}(\eta)\, \dslash\eta=\int_{\R^{2n}}\esp^{i(\phi(t,x,\eta)-y\eta)}b(t,x,\eta)v(y)\, dy\, \dslash\eta
\]
for $v\in\Gc(\R^n)$ and a generalised symbol $b$. Indeed, for giving a meaning to the oscillatory integral on the right-hand side (by means of an operator $L_\phi$ to be used in the integration by parts), it is sufficient that the following bound from below
\beq
\label{new_phase}
|\nabla_x\phi_\eps'(t,x,\eta)|^2+|\nabla_y\phi_\eps'(t,x,\eta)|^2+|\eta|^2|\nabla_\eta\phi_\eps'(t,x,\eta)|^2\ge \lambda_\eps|\eta|^2
\eeq
holds for all $t\in[-T,T]$, $(x,y,\eta)\in\R^{3n}$ and $\eps\in(0,1]$, with $\phi'(t,x,y,\eta)=\phi(t,x,\eta)-y\eta$ and  some strictly non-zero net $\lambda_\eps$. The condition \eqref{new_phase} is trivially satisfied by the phase function in Proposition \eqref{prop_phase_pseudo}. Indeed, for
\[
\phi'(t,x,y,\eta)=(x-y)\eta+\int_{0}^t a_1(s,\eta)\, ds
\]
one has $|\nabla_x\phi_\eps'(t,x,\eta)|^2=|\eta|^2$.

(ii) Easy computations at the level of representatives show that the Fourier integral operator with phase function $\phi(t,x,\eta)=x\eta+\int_{0}^t a_1(s,\eta)\, ds$ and symbol $b(t,x,\eta)\in\G_{\Cinf([-T,T], S^m(\R^{2n}))}$ maps $\Gc(\R^n)$ (or $\GS(\R^n)$) into $\G_{\Cinf([-T,T], \S(\R^n))}$. Recalling the embedding $\Gc(\Rìn)\subseteq\GS(\R^n)\subseteq \G_{2,2}(\R^n)$ (see \cite{Garetto:05b}) we have
\[
F_{\phi}(b):\GS(\R^n)\to \G_{\Cinf([-T,T], \S(\R^n))}\to\G_{2,2}((-T,T)\times\R^n)
\]
and
\[
a_1(t,D_x)F_\phi(b)+a_0(t,x,D_x)F_\phi(b): \GS(\R^n)\to \G_{\Cinf([-T,T], \S(\R^n))}\to\G_{2,2}((-T,T)\times\R^n).
\]
 In addition, if we work with slow scale generalised phase functions and symbols then the previous mapping properties hold between $\GSinf(\R^n)$ and $\G^\infty_{2,2}((-T,T)\times\R^n)$.
\end{remark}
Theorem 5.10 in \cite{Garetto:ISAAC07} concerning the composition of a generalised pseudodifferential operator with a generalised Fourier integral operator can be easily adapted to Fourier integral operators with a phase function as above. This means that under the slow scale assumptions of Theorem \ref{theo_first_order_pseudo}$(iii)$, for any $b\in\G^{\ssc}_{\Cinf{([-T,T], S^m(\R^{2n}))}}$ and $v\in\Gc(\R^n)$, we have
\beq
\label{form_comp_MO}
a_1(t,D_x)F_\phi(b)v+a_0(t,x,D_x)F_\phi(b)v=\int_{\R^n}\esp^{i\phi(t,x,\eta)}h(t,x,\eta)\widehat{v}(\eta)\, \dslash\eta + r(t,x,D_x)v,
\eeq
where $h$ and $r$ are generalised symbols of order $m+1$ and $-\infty$, respectively, $v\in\GS(\R^n)$ and the equality is intended in $\G_{2,2}((-T,T)\times\R^n)$. If we take slow scale regular phase functions and symbols then $h$ and $r$ are slow scale regular as well and $r(t,x,D_x)v\in\G^\infty_{2,2}((-T,T)\times\R^n)$. The symbol $h$ has the following asymptotic expansion (as defined in \cite{Garetto:ISAAC07}):
\begin{multline*}
h(t,x,\eta)\sim \sum_{\alpha\in\N^n}h_\alpha(t,x,\eta)=\sum_{\alpha\in\N^n}\frac{\partial^\alpha_\eta(a_1(t,\eta)+a_0(t,x,\eta))}{\alpha!}D^\alpha_z\big(\esp^{i\overline{\phi(t,z,x,\eta)}}b(t,z,\eta)\big)|_{z=x}\\
=\sum_{\alpha\in\N^n}\frac{\partial^\alpha_\eta(a_1(t,\eta)+a_0(t,x,\eta))}{\alpha!}D^\alpha_x b(t,x,\eta).
\end{multline*}
\subsection{The transport equations and the generalised symbol}
\label{subsec_transport_pseudo}
We proceed in the construction of a generalised FIO parametrix for the Cauchy problem \eqref{CP_pseudo_hom} by looking for a symbol $b\in\G^{\ssc}_{\Cinf{([-T,T], S^0(\R^{2n}))}}$ given by the asymptotic expansion $\sum_{k\in\N} b_k$, $b_k\in\G^{\ssc}_{\Cinf{([-T,T], S^{-k}(\R^{2n}))}}$. The symbols $b_k$ will solve some specific transport equations and will determine $b$ such that
\[
D_t F_\phi(b)-a_1(t,D_x)F_\phi(b)-a_0(t,x,D_x)F_\phi(b)
\]
is a $\Ginf_{2,2}$-regularising operator. We begin by observing that by the composition formula \eqref{form_comp_MO},
\[
a_1(t,D_x)F_\phi(b)v+a_0(t,x,D_x)F_\phi(b)v=F_\phi(h)u+r(t,x,D_x)v,\qquad\quad v\in\Gc(\R^n),
\]
where $r$ has order $-\infty$ and is slow scale regular when $a_1$ and $a_0$ are slow scale regular as well. In other words we can write
\[
D_t F_\phi(b)v-a_1(t,D_x)F_\phi(b)v-a_0(t,x,D_x)F_\phi(b)v=D_tF_\phi(b)v-F_\phi(h)v-r(t,x,D_x)v.
\]
Note that taking a suitable cut-off function $\psi$ we can write the action of $r(t,x,D_x)$ on $v$ as the integral operator
\[
\int_{\R^{2n}}\psi(y)k_r(t,x,y)v(y)\,dy,
\]
with $\psi(y)k_r(t,x,y)\in\G^\infty_{2,2}((-T,T)\times\R^{2n})$ rapidly decreasing with respect to $x$ and $y$. Then the generalised function $\int_{\R^{2n}}\psi(y)k_r(t,x,y)v(y)\,dy$ belongs to $\G^\infty_{2,2}((-T,T)\times\R^n)$. We recall that $\phi$ solves the eikonal equation \eqref{eikonal_pseudo} of the previous subsection. Hence, by making use of the asymptotic expansion of $h$ and $b$ written above and by collecting the terms with the same order we obtain the following transport equations:
\beq
\label{transport_pseudo}
\begin{split}
D_t b_0&=\sum_{|\alpha|=1}\frac{\partial^\alpha_\eta a_1(t,\eta)}{\alpha!}D^\alpha_x b_0+a_0b_0,\\
D_t b_{-1}&=\sum_{|\alpha|=1}\frac{\partial^\alpha_\eta a_1(t,\eta)}{\alpha!}D^\alpha_x b_{-1}+a_0b_{-1}+\sum_{|\alpha|=1}\frac{\partial^\alpha_\eta a_0(t,x,\eta)}{\alpha!}D^\alpha_x b_0+\sum_{|\alpha|=2}\frac{\partial^\alpha_\eta a_1(t,\eta)}{\alpha!}D^\alpha_x b_0,\\
\vdots&\\
D_t b_{-k}&=\sum_{|\alpha|=1}\frac{\partial^\alpha_\eta a_1(t,\eta)}{\alpha!}D^\alpha_x b_{-k}+a_0b_{-k}+f_{-k},
\end{split}
\eeq
where $f_{-k}\in\G^{\ssc}_{\Cinf{([-T,T], S^{-k}(\R^{2n}))}}$. From Proposition \ref{prop_FIO_formula} and Theorem \ref{theo_FIO_formula_nh} of the previous section we deduce the following statement on transport equations with generalised symbols as coefficients and initial data.
\begin{theorem}
\label{theo_transport_pseudo}
The Cauchy problem
\begin{align}
\label{transport_sym}
 D_t s &=\sum_{|\alpha|=1}\frac{\partial^\alpha_\eta a_1(t,\eta)}{\alpha!}D^\alpha_x s+ a_0(t,x,\eta) s +f\\
 s(0,\cdot,\cdot)&=s_0,
\end{align}
where
\begin{itemize}
\item[(i)] $m\in\R$, $f\in\G_{\Cinf([-T,T], S^{m}(\R^{2n}))}$ and $s_0\in\G_{\Cinf([-T,T], S^{m}(\R^{2n}))}$,
\item[(ii)] $a_1\in \G_{\Cinf([-T,T], S^1(\R^{2n}))}$ and $a_0\in \G_{\Cinf([-T,T], S^0(\R^{2n}))}$,
\item[(iii)] $a_0$ is of log-type,
\end{itemize}
has a solution $s\in\G_{\Cinf([-T,T], S^{m}(\R^{2n}))}$.
If we replace $\G$ with $\G^\ssc$ in $(i)$, $(ii)$, $(iii)$ and the assumption of log-type with slow scale log-type then $s\in\G^\ssc_{\Cinf([-T,T], S^{m}(\R^{2n}))}$.
\end{theorem}
\begin{proof}
By applying Theorem \ref{theo_FIO_formula_nh} to the Cauchy problem \eqref{transport_sym} we have the following solution formula:
\[
s(t,x,\eta)=b(t,x,\eta)\biggl(s_{0}(t,\gamma(t,x,\eta),\eta)+i\int_{0}^t\frac{f(\tau,\gamma(t,x,\eta,\tau),\eta)}{b(\tau,\gamma(t,x,\eta,\tau),\eta)}\, d\tau\biggr)
\]
where, for $i=1,...,n$,
\[
\begin{split}
\gamma_i(t,x,\eta)&=x_i+\int_0^t\partial_{\eta_i}a_{1}(\sigma,\eta)\, d\sigma,\\
\gamma_i(t,x,\eta,\tau)&=x_i+\int_0^t\partial_{\eta_i}a_{1}(\sigma,\eta)\, d\sigma-\int_0^\tau\partial_{\eta_i}a_{1}(\sigma,\eta)\, d\sigma,\\
b(t,x,\eta)&=\esp^{i\int_{0}^t a_0(s,\gamma(t,x,\eta,s),\eta)\, ds}.
\end{split}
\]
One easily checks that $b(t,x,\eta)$ and $s_{0}(t,\gamma_1(t,x,\eta),...,\gamma_2(t,x,\eta),\eta)$ are elements of the  generalised symbol spaces $\G_{\Cinf([-T,T], S^{0}(\R^{2n}))}$ and $\G_{\Cinf([-T,T], S^{m}(\R^{2n}))}$, respectively, and that
\[
\int_{0}^t\frac{f(\tau,\gamma(t,x,\eta,\tau),\tau),\eta)}{b(\tau,\gamma(t,x,\eta,\tau),\tau),\eta)}d\tau
=\int_0^t\esp^{-i\int_0^\tau a_0(s,\gamma(t,x,\eta,s),\eta)\, ds}f(\tau,\gamma(t,x,\eta,\tau),\tau),\eta)\, d\tau
\]
belongs to $\G_{\Cinf([-T,T], S^{m}(\R^{2n}))}$. From this it follows that $s(t,x,\eta)$ is a generalised symbol of order $m$ satisfying the initial condition $s(0,x,\eta)=s_0(x,\eta)$.
\end{proof}
We can now go back to the system of equations \eqref{transport_pseudo}. With the help of Theorem \ref{theo_transport_pseudo} we can solve the equations in \eqref{transport_pseudo}:
\begin{proposition}
\label{prop_transport_pseudo}
If $a_1$ and $a_0$ are slow scale regular generalised symbols of order $1$ and $0$, respectively, satisfying the assumptions of Theorem \ref{theo_first_order_pseudo}$(iii)$ then there exist symbols $b_{-j}\in\G^\ssc_{\Cinf([-T,T], S^{-j}(\R^{2n}))}$, $j\in\N$, which solve the transport equations \eqref{transport_pseudo} and such that $b_0(0,x,\eta)=1$ and $b_{-j}(0,x,\eta)=0$ for all $j>0$.
\end{proposition}
\subsection{Construction of a generalised FIO parametrix and solution formula for the Cauchy problem}
\label{subsec_param_FIO}
We now have all the tools for constructing a generalised FIO parametrix for the Cauchy problem \eqref{CP_pseudo_hom}. Combining Theorem \ref{theo_first_order_pseudo} with Proposition \ref{prop_phase_pseudo} and Proposition \ref{prop_transport_pseudo} we obtain the following statement.
\begin{theorem}
\label{theo_FIO_par}
Let $u_0\in\Gc(\R^n)$ and let $a_1$ and $a_0$ fulfill the assumptions of Theorem \ref{theo_first_order_pseudo}$(iii)$ on $a_1$ and $a_0$. Then there exists a generalised Fourier integral operator $F_\phi(b)$ with slow scale phase function and symbol $b$ of order $0$ such that
\begin{align*}
D_tF_\phi(b)u_0&=a_1(t,D_x)F_\phi(b)u_0+a_0(t,x,D_x)F_\phi(b)u_0+r(t,x,D_x)u_0,\\
F_\phi(b)u_0(0,\cdot)&=u_0+r_0(t,D_x)u_0,
\end{align*}
where $r\in\G^{\ssc}_{\Cinf([-T,T],S^{-\infty}(\R^{2n}))}$ and $r_0\in\G^{\ssc}_{S^{-\infty}(\R^{2n})}$.
\end{theorem}
\begin{corollary}
\label{corol_theo_FIO_par}
Let the hypotheses of Theorem \ref{theo_first_order_pseudo}$(iii)$ be satisfied. Then the solution $u\in\G_{2,2}((-T,T)\times\R^n)$ of the Cauchy problem
\[
\begin{split}
 D_t u &=a_1(t,D_x)u+a_0(t,x,D_x)u \\
 u(0,\cdot)&=u_0\in\Gc(\R^n),
\end{split}
\]
is equal to $F_\phi(b)u_0$ modulo $\G^{\infty}_{2,2}((-T,T)\times\R^n))$.
\end{corollary}
\begin{proof}
From Theorem \ref{theo_FIO_par} we have that $v=F_\phi(b)u_0-u\in\G_{2,2}((-T,T)\times\R^n)$ solves the Cauchy problem
\[
\begin{split}
 D_t v &=a_1(t,D_x)v+a_0(t,x,D_x)v+r(t,x,D_x)u_0, \\
 v(0,\cdot)&=r_0(x,D_x)u_0,
\end{split}
\]
where $r(t,x,D_x)u_0\in\G^{\infty}_{2,2}((-T,T)\times\R^n)$ and $r_0(x,D_x)u_0\in\G^{\infty}_{2,2}(\R^n)$. Theorem \ref{theo_first_order_pseudo} yields the $\Ginf$-regularity of $v$, i.e., $v\in\G^{\infty}_{2,2}((-T,T)\times\R^n)$.
\end{proof}

\section{Microlocal investigation of the solution of a generalised hyperbolic Cauchy problem}
\label{sec_micro}
This section provides a microlocal investigation of the solution $u\in\G(\R^{n+1})$ of the hyperbolic Cauchy problem studied in Section \ref{sec_PDE} and Section \ref{sec_pseudo}. First we will concentrate on the microlocal properties of $u$, viewed as a generalised function in both the variables $t$ and $x$ ($\WF_{\Ginf}u$), and second we will fix $t$ and investigate the generalised function $u(t,\cdot)\in\G(\R^n)$ microlocally ($\WF_{\Ginf}u(t,\cdot)$). Since $u$ can be written as the action of a generalised Fourier integral operator $F_\phi(b)$ on the initial data $u_0\in\Gc(\R^n)$ we focus our attention on the microlocal properties of generalised Fourier integral operators. We begin with some abstract theoretical results that we will finally apply to the special case of $u=F_\phi(b)u_0$ under suitable assumptions on the phase function $\phi$.
\subsection{Microlocal properties of generalised Fourier integral operators: the wave front set}
We begin by recalling some results obtained in \cite{GHO:09}. Let $\Om$ be an open subset of $\R^n$. As a preliminary step we observe that when $\phi\in\wt{\Phi}(\Om\times\R^p)$, $U\subseteq\overline{U}\Subset\Om$, $\Gamma\subseteq\R^n\setminus 0$, $V\subseteq \Om\times\R^p\setminus 0$, then
\[
\mathop{{\rm{Inf}}}\limits_{\substack{y\in U, \xi\in \Gamma\\ (y,\theta)\in V}}\frac{|\xi-\nabla_y\phi(y,\theta)|}{|\xi|+|\theta|}:=\biggl[\biggl(\inf_{\substack{y\in U, \xi\in \Gamma\\ (y,\theta)\in V}}\frac{|\xi-\nabla_y\phi_\eps(y,\theta)|}{|\xi|+|\theta|}\biggr)_\eps\biggr]
\]
is a well-defined element of $\wt{\C}$.

In the sequel $W^\ssc_{\phi,a}$ denotes the set of all points $(x_0,\xi_0)\in\CO{\Om}$ with the property that for all relatively compact open neighborhoods $U(x_0)$ of $x_0$, for all open conic neighborhoods $\Gamma(\xi_0)\subseteq \R^n\setminus 0$ of $\xi_0$, for all open conic neighborhoods $V$ of {\rm{cone\,supp}}\,$a\cap C^\ssc_\phi$ such that $V\cap (U(x_0)\times\R^p\setminus 0)\neq\emptyset$ the generalised number
\beq
\label{gen_num_inv}
\mathop{{\rm{Inf}}}\limits_{\substack{y\in U(x_0), \xi\in \Gamma(\xi_0)\\ (y,\theta)\in V\cap( U(x_0)\times\R^p\setminus 0)}}\frac{|\xi-\nabla_y\phi(y,\theta)|}{|\xi|+|\theta|}
\eeq
is not slow scale-invertible.
\begin{theorem}
\label{theo_GHO}
Let $I_\phi(a)$ be the functional in $\LL(\Gc(\Om),\wt{\C})$ given by
\[
v\to\int_{\Om\times\R^n}\esp^{i\phi(y,\xi)}a(y,\xi)v(y)\, dy\, \dslash\xi,
\]
where $\phi$ is a slow scale generalised phase function and $a$ is a regular symbol of order $m$. The $\Ginf$-wave front set of $I_\phi(a)$ is contained in the set $W^\ssc_{\phi,a}$.
\end{theorem}
The previous theorem can clearly be stated for a slow scale generalised phase function $\phi(x,y,\xi)$ (in the variable $(y,\xi)$) and a regular amplitude $a(x,y,\xi)$ on $\Om'\times\Om\times\R^p$, with an open subset $\Om'$ of $\R^{n'}$. In this case
\[
I_\phi(a)v=\int_{\Om'\times\Om\times\R^n}\esp^{i\phi(x,y,\xi)}a(x,y,\xi)v(x,y)\, dx\, dy\, \dslash\xi,
\]
is the kernel of the generalised Fourier integral operator
\[
A:\Gc(\Om)\to \G(\Om'):u\to \int_{\Om\times\R^n}\esp^{i\phi(x,y,\xi)}a(x,y,\xi)u(y)\, dy\, \dslash\xi.
\]
We will therefore use the notation $K_A$ for the functional $I_\phi(a)$. The next theorem relates $\WF_{\Ginf}Au$ with $\WF_{\Ginf}K_A$ and $\WF_{\Ginf}u$.

We recall that given $E\subseteq\Om'\times\Om\times\R^{n'}\times\R^n$, $E'\subseteq (\Om'\times\R^{n'})\times(\Om\times\R^n)$ is the set of all $(x,\xi,y,\eta)$ such that $(x,y,\xi,-\eta)\in E$. A subset $E'$ defines a relation $E'\circ G$ of $\Om\times\R^n$ with $\Om'\times\R^{n'}$ when $G\subseteq \Om\times\R^n$. In detail,
\[
E'\circ G:=\{(x,\xi)\in\Om'\times\R^{n'}:\ \exists (y,\eta)\in G,\ (x,\xi,y,\eta)\in E'\}.
\]
\begin{theorem}
\label{theo_ChazP}
Let $\phi(x,y,\xi)$ be a slow scale generalised phase function in the variable $(y,\xi)$ and $a(x,y,\xi)$ be a regular amplitude on $\Om'\times\Om\times\R^p$. Let $A$ be the corresponding Fourier integral operator with kernel $K_A\in\LL(\G_c(\Om'\times\Om),\wt{\C})$ and $u$ a generalised function in $\Gc(\Om)$. If
\beq
\label{rel_cond}
\WF'_{\Ginf}K_A\circ\WF_{\Ginf} u\subseteq \CO{\Om'}
\eeq
then
\beq
\label{chain_ChazP}
\WF_{\Ginf}Au\subseteq \big(\WF'_{\Ginf}K_A\circ\WF_{\Ginf} u\big) \cup \big(\WF'_{\Ginf}K_A\circ(\singsupp_{\Ginf} u\times\{0\})\big),
\eeq
where the right-hand side is a conic subset of $\CO{\Om'}$.
\end{theorem}
\begin{proof}
Let $L$ be a neighborhood of $\singsupp_{\Ginf}u$. Let $\Lambda$ be the right-hand side of the assertion \eqref{chain_ChazP} with $\singsupp_{\Ginf} u$ replaced by $L$. It is clear that $\Lambda$ is conic and the inclusion $\Lambda\subseteq \CO{\Om'}$ follows from \eqref{rel_cond}.

We will prove that if $(x_0,\xi_0)\not\in\Lambda$ then $(x_0,\xi_0)\not\in\WF_{\Ginf} Au$. From this it will follow $\WF_{\Ginf} Au\subseteq \Lambda$. At this point since $L$ is chosen arbitrarily and $u$ has compact support we get the desired inclusion \eqref{chain_ChazP}.

Let $K_1$ be a neighborhood of $x_0$, and $K_2=\pi_\Om(\singsupp_{\Ginf} K_A\cap (K_1\times L))$. From the assumptions on phase function and amplitude we have that $A$ maps $\Ginf$ into $\Ginf$. It follows that for a cut-off function $\psi$ which is identically $1$ on a neighborhood of $\singsupp_{\Ginf}u$ we can write $\singsupp_{\Ginf}Au=\singsupp_{\Ginf}A(\psi u)$. If $K_2=\emptyset$ then taking $\psi$ such that $\supp(\psi u)\subseteq L$ we have that $Au|_{K_1}$ can be written as $\int K_A(x,y)\psi u(y)\, dy$, where $K_A(x,y)$ is $\Ginf$-regular. Hence $Au|_{K_1}\in\Ginf$ and $(x_0,\xi)\not\in\WF_{\Ginf} Au$ for each $\xi\neq 0$.

We assume that $K_2$ is not empty and for $y_0\in K_2$ we set
\[
\begin{split}
\Sigma_0 &=\{(\xi,\eta):\ (x_0,y_0,\xi,\eta)\in\WF_{\Ginf}K_A\},\\
\Gamma_0 &=\{ \eta:\ (y_0,\eta)\in\WF_{\Ginf}u\}.
\end{split}
\]
We claim that there exist:
\begin{itemize}
\item[-] a conic open neighborhood $\wt{\Sigma}_0\subseteq\R^{n'+n}\setminus 0$ of $\Sigma_0$,
\item[-] a conic open neighborhood $\wt{\Gamma}_0\subseteq\R^n\setminus 0$ of $\Gamma_0$,
\item[-] a conic neighborhood $W\subseteq\R^{n'}\setminus 0$ of $\xi_0$,
\item[-] a number $\delta>0$ and neighborhoods $V_1$ and $V_2$ of $x_0$ and $y_0$, respectively,
\end{itemize}
such that
\beq
\label{incl_1}
W\cap\big({\wt{\Sigma}_0}'\circ\wt{\Gamma}_0\big)=\emptyset,
\eeq
\beq
\label{incl_2}
V_1\times V_2\times\{(\xi,\eta):\ \xi\in W,\ |\eta|\le\delta|\xi|\}\cap\WF_{\Ginf}K_A=\emptyset,
\eeq
\beq
\label{incl_3}
\wt{\Sigma}_0\cap\{(\xi,\eta):\ -\eta\in\wt{\Gamma}_0,\ |\xi|\le \delta|\eta|\}=\emptyset.
\eeq
But $(x_0,\xi_0)\not\in\Lambda$ implies $(x_0,\xi_0)\not\in\WF'_{\Ginf}K_A\circ\WF_{\Ginf} u$. Therefore, since
${\Sigma_0}$, ${\Gamma_0}$ and ${\Sigma_0}\circ{\Gamma_0}$ are closed and conic there exist neighborhoods $W$, $\wt{\Sigma}_0$ and $\wt{\Gamma}_0$ as above satisfying \eqref{incl_1}. If $(x_0,\xi_0)\not\in\Lambda$ then $(x_0,y_0,\xi_0,0)\not\in\WF_{\Ginf}K_A$. Indeed, $(x_0,y_0,\xi_0,0)\in\WF_{\Ginf}K_A$ implies $(x_0,\xi_0)\in\WF_{\Ginf}K_A'\circ(L\times\{0\})\subseteq\Lambda$. It follows that there exist neighborhoods $V_1$ and $V_2$ as above and $\delta>0$ such that \eqref{incl_2} holds. Finally, we have $\{(0,\eta):\, -\eta\in\Gamma_0\}\cap\Sigma_0=\emptyset$. Indeed, if $(x_0,0,y_0,\eta)\in\WF_{\Ginf}K_A$ and $(y_0,-\eta)\in\WF_{\Ginf}u$ then $(x_0,0)\in\WF'_{\Ginf}K_A\circ\WF_{\Ginf} u$. This is absurd since $\WF'_{\Ginf}K_A\circ\WF_{\Ginf} u\subseteq \CO{\Om'}$. This assertion yields \eqref{incl_3} for a suitable choice of $\wt{\Sigma}_0$ and $\wt{\Gamma}_0$. By shrinking $V_1$ and $V_2$ we can also impose
\beq
\label{incl_4}
V_1\subseteq K_1,\qquad (V_2\times(\wt{\Gamma}_0)^{\rm{c}})\cap\WF_{\Ginf}u=\emptyset,\qquad (V_1\times V_2\times(\wt{\Sigma}_0)^{\rm{c}})\cap\WF_{\Ginf}K_A=\emptyset,
\eeq
where the complements are intended in $\R^{n}\setminus 0$ and $\R^{n'+n}\setminus 0$, respectively.

Letting $\alpha\in\Cinfc(V_1)$ identically $1$ near $x_0$ we have to prove that $\widehat{\alpha Au}$ is $\Ginf_{\S\hskip-2pt,0}$-regular on a certain conic neighborhood of $\xi_0$. Let $\psi\in\Cinfc(\Om)$ identically $1$ on a neighborhood of $\singsupp_{\Ginf}u$ and with $\supp\psi\subseteq L$. Since $\alpha A((1-\psi)u)\in\Gcinf(\Om')$ we have $\widehat{\alpha Au}=\widehat{\alpha A(\psi u)}$ modulo $\GSinf$. We write
\[
\widehat{\alpha A\psi u}(\xi)=\lara{K_A,\alpha(x)\esp^{-ix\xi}\psi(y)u(y)}.
\]
From $K_2=\pi_\Om(\singsupp_{\Ginf} K_A\cap (K_1\times L))$ we have that we can restrict our attention to $y\in K_2$. Taking a partition of unity in a neighborhood of $K_2$ (based on the neighborhoods $V_2$ with the property above) we observe that $\widehat{\alpha A\psi u}(\xi)$ is a finite sum of terms of the form
\[
\lara{fK_A,\alpha(x)\esp^{-ix\xi}\beta(y)u(y)},
\]
where $f\in\Cinfc(V_1\times V_2)$ and $\beta\in\Cinfc(V_2)$. Applying the Fourier and inverse Fourier transform we can write
\[
\lara{fK_{A},\alpha(x)\esp^{-ix\xi}\beta(y)u(y)}=\int_{\R^{n'+n}}\widehat{fK_{A}}(\theta,\eta)\widehat{\alpha}(\xi-\theta)\widehat{\beta u}(-\eta)\, \dslash\theta\, \dslash\eta.
\]
Now, from \eqref{incl_4} we have that the tempered generalised function $\widehat{fK_{A}}$ is $\Ginf_{\S\hskip-2pt,0}$-regular outside $\wt{\Sigma}_0$ and that $\widehat{\alpha}(\xi-\theta)\widehat{\beta u}(-\eta)$ is $\Ginf_{\S\hskip-2pt,0}$-regular with respect to $(\theta,\eta)$ outside $\{(\theta,\eta):\ -\eta\in\wt{\Gamma}_0,\ |\theta|\le \delta|\eta|\}$. This fact combined with \eqref{incl_3} guarantees that the integral above is absolutely convergent. Note that
\[
|\widehat{\alpha}(\xi-\theta)|\le c_N(1+|\xi-\theta|)^{-N},
\]
for arbitrary $N\in\N$ and that if $W'$ is a sufficiently small conic open neighborhood of $\xi_0$ with $\overline{W'}\subseteq W$ then there exists $c>0$ such that $|\xi-\theta|\ge c(|\xi|+|\theta|)$ for all $\xi\in W'$ and $\theta\in W^{\rm{c}}$. Hence, the estimate
\[
|\widehat{\alpha}(\xi-\theta)|\le c'_N(1+|\xi|+|\theta|)^{-N}
\]
holds for all $\xi\in W'$ and $\theta\in W^{\rm{c}}$. It follows that
\beq
\label{first}
\int_{\theta\not\in W}\widehat{fK_{A}}(\theta,\eta)\widehat{\alpha}(\xi-\theta)\widehat{\beta u}(-\eta)\, \dslash\theta\, \dslash\eta\in\Ginf_{\S\hskip-2pt,0}(W').
\eeq
In order to understand the integration over $W\times\R^n$ we begin by proving that $
\widehat{fK_{A}}(\theta,\eta)\widehat{\beta u}(-\eta)$ is $\Ginf_{\S\hskip-2pt,0}$-regular in $(\theta,\eta)$ for $\theta\in W$ and $\eta\in\R^n$. From \eqref{incl_2} we have that $\widehat{fK_{A}}$ is $\Ginf_{\S\hskip-2pt,0}$-regular in $(\theta,\eta)$ for $\theta\in W$ and $|\eta|\le \delta|\theta|$. If $|\eta|\ge \delta|\theta|$ and $-\eta\not\in\wt{\Gamma}_0$ \eqref{incl_4} implies that $\widehat{\beta u}$ is $\Ginf_{\S\hskip-2pt,0}$-regular in $\eta$ and therefore also in $(\theta,\eta)$. Finally, when $\theta\in W$ and $-\eta\in\wt{\Gamma}_0$ then $(\theta,\eta)\not\in\wt{\Sigma}_0$ from \eqref{incl_1}. Since $\widehat{fK_{A}}$ is $\Ginf_{\S\hskip-2pt,0}$-regular outside $\wt{\Sigma}_0$ we conclude that $\widehat{fK_{A}}(\theta,\eta)\widehat{\beta u}(-\eta)$ has the same regularity property. At this point it is clear that the generalised function given by
\beq
\label{second}
\int_{\theta\in W}\widehat{fK_{A}}(\theta,\eta)\widehat{\alpha}(\xi-\theta)\widehat{\beta u}(-\eta)\, \dslash\theta\, \dslash\eta
\eeq
is $\Ginf_{\S\hskip-2pt,0}$-regular on $W$. A combination of \eqref{first} with \eqref{second} yields $\widehat{\alpha Au}\in\Ginf_{\S\hskip-2pt,0}(W')$. This means $(x_0,\xi_0)\not\in\WF_{\Ginf}Au$ and completes the proof.
\end{proof}
We recall that the $\Ginf$-microsupport of a generalised symbol $a$ on $\Om\times\R^n$ is the complement of the set of points $(x_0,\xi_0)$ with the following property: there exist a representative $(a_\eps)_\eps$ of $a$, a relatively compact open neighborhood $U$ of $x_0$, a conic neighborhood $\Gamma\subseteq\R^n\setminus 0$ of $\xi_0$ and a natural number $N\in\N$ such that
\[
\forall m\in\R\, \forall\alpha,\beta\in\N^n\, \exists c>0\, \exists\eta\in(0,1]\, \forall(x,\xi)\in U\times\Gamma\, \forall\eps\in(0,\eta]\qquad
|\partial^\alpha_\xi\partial^\beta_x a_\eps(x,\xi)|\le c\lara{\xi}^m\eps^{-N}.
\]
It is denoted by $\mu\supp_{\Ginf} a$.
\begin{remark}
\label{rem_pseudo}
Condition \eqref{rel_cond} is satisfied when $A=a(x,D)$ is a generalised pseudodifferential operator with regular symbol. In this case we have from Remark 5.15 in \cite{GHO:09} that if $(x,y,\xi,\eta)\in \WF_{\Ginf}K_A=\WF_{\Ginf}K_{a(x,D)}$ then $x=y$ and $\eta=-\xi$ with $\xi\neq 0$ and $(x,\xi)\in\mu\supp_{\Ginf}a$. Hence
\[
\WF'_{\Ginf}K_{a(x,D)}\circ\WF_{\Ginf}u\subseteq\{(x,\xi):\, (x,x,\xi,-\xi)\in\WF_{\Ginf}K_{a(x,D)}\}\cap\WF_{\Ginf}u
\]
which implies $\WF'_{\Ginf}K_{a(x,D)}\circ\WF_{\Ginf} u\subseteq \CO{\Om}$. Since
\[
\WF'_{\Ginf}K_{a(x,D)}\circ(\singsupp_{\Ginf} u\times\{0\})=\emptyset
\]
Theorem \ref{theo_ChazP} yields, for $u\in\Gc(\Om)$,
\begin{multline*}
\WF_{\Ginf}a(x,D)u\subseteq \WF'_{\Ginf}K_{a(x,D)}\circ\WF_{\Ginf} u\\
\subseteq\{(x,\xi):\, (x,x,\xi,-\xi)\in\WF_{\Ginf}K_{a(x,D)}\}\cap\WF_{\Ginf}u\subseteq \mu\supp_{\Ginf}a\cap\WF_{\Ginf}u.
\end{multline*}
Note that the same inclusion was already obtained in \cite[Theorem 3.6]{GH:05}.
\end{remark}
We now consider a special class of generalised Fourier integral operators which fulfill \eqref{rel_cond} and for them we rephrase Theorem \ref{theo_ChazP}. This requires some preliminary results concerning the set $C_\phi$ of singular points and the set $W^{\ssc}_{\phi,a}$ defined before Theorem \ref{theo_GHO}.
\begin{proposition}
\label{prop_phi_C1}
Let $[(\phi)_\eps]\in\wt{\Phi}(\Om\times\R^p)$ and $\phi\in{\mathcal{C}}^1(\Om\times\R^p\setminus 0)$ such that $\lim_{\eps\to 0}\phi_\eps=\phi$ in ${\mathcal{C}}^1(\Om\times\R^p\setminus 0)$. Then
\begin{itemize}
\item[(i)] $C^\ssc_{[(\phi_\eps)_\eps]}\subseteq\{(x,\xi)\in\CO{\Om}:\ \nabla_\xi\phi(x,\xi)=0\}$,
\item[(ii)] for any generalised symbol $a$,
\[
W^{\ssc}_{[(\phi)_\eps],a}\subseteq\{(x,\nabla_x\phi(x,\theta)):\, \theta\neq 0,\ (x,\theta)\in{\rm{cone\,supp}}\, a,\ \nabla_\theta\phi(x,\theta)=0\}.
\]
\end{itemize}
\end{proposition}
\begin{proof}
$(i)$ From the limit property we obtain that $\sup_{x\in K, |\xi|=1}|\nabla_\xi\phi_\eps(x,\xi)-\nabla_\xi\phi(x,\xi)|$ tends to $0$ for any compact subset $K$ of $\Om$. If for some $(x_0,\xi_0)$, $\nabla_\xi\phi(x_0,\xi_0)\neq 0$ then $|\nabla_\xi\phi(x,\xi)|\ge c$ on a relatively compact neighborhood $U(x_0)$ and on a conic neighborhood $\Gamma(\xi_0)$. Since for all $n\in\N$ there exists $\eps_n\in(0,1]$ such that
\[
\sup_{x\in U(x_0), \xi\in\Gamma(\xi_0)}|\nabla_\xi\phi_\eps(x,\xi)-\nabla_\xi\phi(x,\xi)|\le \frac{1}{n},
\]
for all $\eps\in(0,\eps_n)$, we get
\[
|\nabla_\xi\phi_\eps(x,\xi)|\ge |\nabla_\xi\phi(x,\xi)|-|\nabla_\xi\phi_\eps(x,\xi)-\nabla_\xi\phi(x,\xi)|\ge c-\frac{1}{n}
\]
for all $x\in U(x_0)$, $\xi\in\Gamma(\xi_0)$, $\eps\in(0,\eps_n)$. Choosing $n$ large enough we have that the net $$\inf_{x\in U(x_0), \xi\in\Gamma(\xi_0)}|\nabla_\xi\phi_\eps(x,\xi)|$$ is slow scale invertible, that is $(x_0,\xi_0)\not\in C^\ssc_{[(\phi_\eps)_\eps]}$.

$(ii)$ If $(x_0,\xi_0)\in W^{\ssc}_{[(\phi)_\eps],a}$ then we find sequences of neighborhoods $U_n(x_0)$, $\Gamma_n(\xi_0)$, and $V_n({\rm{cone\, supp}}\, a\cap C^\ssc_{[(\phi_\eps)_\eps]})$ and sequences of points $x_n\in U_n(x_0)$, $\xi_n\in\Gamma_n(\xi_0)$, $(x_n,\theta_n)\in V_n$ such that
\[
|\xi_n-\nabla_x\phi_{\eps_n}(x_n,\theta_n)|\le |\log \eps_n|^{-1}(|\xi_n|+|\theta_n|),
\]
where $\eps_n$ tends to $0$. It is not restrictive to assume that $\theta_n$ has norm $1$. Hence,
\beq
\label{ineq_lim}
|\xi_n-\nabla_x\phi_{\eps_n}(x_n,\theta_n)|\le |\log \eps_n|^{-1}(|\xi_n|+1).
\eeq
Passing to subsequences we have that $x_n$ converges to $x_0$ and $\theta_n$ to a certain $\theta_0$ with $|\theta_0|=1$ and $(x_0,\theta_0)\in{\rm{cone\, supp}}\, a\cap C^\ssc_{[(\phi_\eps)_\eps]}$. The first assertion of this proposition implies that $\nabla_\theta\phi(x_0,\theta_0)=0$. The sequence $|\xi_n|$ is bounded. Indeed,
\[
|\xi_n|\le |\xi_n-\nabla_x\phi_{\eps_n}(x_n,\theta_n)|+|\nabla_x\phi_{\eps_n}(x_n,\theta_n)|\le |\log \eps_n|^{-1}(|\xi_n|+1)+|\nabla_x\phi_{\eps_n}(x_n,\theta_n)|,
\]
where, from $\lim_{n\to\infty}|\nabla_x\phi_{\eps_n}(x_n,\theta_n)|=|\nabla_x\phi(x_0,\theta_0)|$ it follows that $|\nabla_x\phi_{\eps_n}(x_n,\theta_n)|\le 1$ for $n$ large enough. Thus,
\[
|\xi_n|(1-|\log \eps_n|^{-1})\le |\log \eps_n|^{-1}+1
\]
results in an upper bound for $|\xi_n|$. As a consequence, passing again to subsequences, we find some $\xi'=\lambda\xi_0$, $\lambda>0$, such that $\xi_n\to \xi'$. Passing to the limit in \eqref{ineq_lim} we deduce $\xi'=\lambda\xi_0=\nabla_x\phi(x_0,\theta_0)$. In conclusion, we have proved that if $(x_0,\xi_0)\in W^{\ssc}_{[(\phi)_\eps],a}$ then there exists $\theta_0'\neq 0$ such that $(x_0,\theta_0')\in {\rm{cone\, supp}}\, a$, $\nabla_\theta\phi(x_0,\theta_0')=0$ and $\xi_0=\nabla_x\phi(x_0,\theta_0')$.
\end{proof}
We are ready to consider a Fourier integral operator with phase function
\[
[(\phi_\eps(x,\xi))_\eps]-y\xi.
\]
From the previous statements we obtain the following microlocal result.
\begin{theorem}
\label{theo_micro_cla}
Let $[(\phi_\eps(x,\xi)-y\xi)_\eps]$ be a slow scale generalised phase function in the variable $(y,\xi)$ and $a(x,y,\xi)$ a regular generalised amplitude on $\Om'\times\Om\times\R^n$. Let $A$ be the corresponding Fourier integral operator with kernel $K_A\in\LL(\G_c(\Om'\times\Om),\wt{\C})$ and $u$ a generalised function in $\Gc(\Om)$. If there exists $\phi\in{\mathcal{C}}^1(\Om'\times\R^n\setminus 0)$ with $\nabla_x\phi(x,\xi)\neq 0$ for $\xi\neq 0$ such that $\phi_\eps\to \phi$ in ${\mathcal{C}}^1(\Om'\times\R^n\setminus 0)$ then
\begin{itemize}
\item[(i)] $\WF'_{\Ginf}K_A\circ\WF_{\Ginf} u\subseteq \CO{\Om'}$,
\item[(ii)] $\WF_{\Ginf}Au\subseteq\{(x,\nabla_x\phi(x,\theta)):\ (x,\nabla_\theta\phi(x,\theta),\theta)\in{\rm{cone\,supp}}\, a,\ (\nabla_\theta\phi(x,\theta),\theta)\in\WF_{\Ginf}u\}.$
\end{itemize}
\end{theorem}
\begin{proof}
An application of Proposition \ref{prop_phi_C1} to the phase function $[(\phi_\eps(x,\xi)-y\xi)_\eps]$ yields
\beq
\label{formula_cla}
WF_{\Ginf}K_A\subseteq\{(x,y,\nabla_x\phi(x,\theta),-\theta):\ \theta\neq0,\ (x,y,\theta)\in{\rm{cone\,supp}}\, a,\ \nabla_\theta\phi(x,\theta)=y\},
\eeq
i.e.,
\[
WF_{\Ginf}K_A\subseteq\{(x,\nabla_\theta\phi(x,\theta),\nabla_x\phi(x,\theta),-\theta):\ \theta\neq0,\ (x,y,\theta)\in{\rm{cone\,supp}}\, a\}.
\]
From the hypothesis on $\nabla_x\phi$ we have that if $(x,\xi)\in\WF'_{\Ginf}K_A\circ\WF_{\Ginf} u$ then $\xi=\nabla_x\phi(x,\theta)$ for some $\theta\neq 0$. It follows that $\xi\neq 0$. From \eqref{formula_cla} and Theorem \ref{theo_ChazP} we easily obtain the inclusion $(ii)$.
\end{proof}
We now apply Theorem \ref{theo_micro_cla} to the case of a generalised Fourier integral operator solving a hyperbolic Cauchy problem as in Section \ref{sec_PDE}. We obtain the following microlocal result.
\begin{proposition}
\label{prop_micro_cla}
Let
\[
u(t,x)=\int_{\R^n}\esp^{i[(\phi_\eps(t,x,\eta))_\eps]}b(t,x,\eta)\widehat{u_0}(\eta)\, \dslash\eta,
\]
where $b$ is a regular generalised symbol of order $0$ and $[(\phi_\eps(t,x,\eta)-y\eta)_\eps]$ is a slow scale generalised phase function in the variable $(y,\eta)$. Assume that there exists $\phi\in{\mathcal{C}}^1(\R\times\R^n\times\R^n\setminus 0)$ with $\nabla_{(t,x)}\phi(t,x,\eta)\neq 0$ for $\eta\neq 0$ such that $\phi_\eps\to \phi$ in ${\mathcal{C}}^1(\R\times\R^n\times\R^n\setminus 0)$. Then
\[
\WF_{\Ginf}u\subseteq \{(t,x,\nabla_t\phi(t,x,\eta),\nabla_x\phi(t,x,\eta)):\ (t,x,\eta)\in {\rm{cone\,supp}}\, b,\ (\nabla_\eta\phi(t,x,\eta),\eta)\in\WF_{\Ginf}u_0\}.
\]
\end{proposition}
\begin{example}
\label{rem_ex_MO}
We demonstrate how Proposition \ref{prop_micro_cla} can be used in a simple example. It concerns a transport equation with discontinuous coefficients and distributional data taken from in \cite[Section 5]{O:07}. In detail, let
\begin{align*}
D_t u(t,x)&=-H(t-1)D_x u(t,x),\\
u(0,x)&=\delta(x),\\
\end{align*}
with $t\in[0,+\infty)$ and $x\in\R$. Here $H$ denotes the Heaviside function and $\delta$ the Dirac measure. Take $\rho\in\D(\R)$ with $\supp\rho\subseteq[-1,1]$, nonnegative, symmetric and with integral equal $1$. Let $(\omega^{-1}(\eps))_\eps$ be a slow scale net with $\lim_{\eps\to 0}\omega(\eps)=0$. The solution $u\in\G([0,+\infty)\times\R)$ of the Cauchy problem can be written, at the level of representatives, as
\[
u_\eps(t,x)=\int_\R\esp^{i(x-\Lambda_\eps(t))\eta}\widehat{\rho_\eps}(\eta)\, \dslash\eta,
\]
where
\[
\Lambda_\eps(t)=\begin{cases}
0, & 0\le t\le 1-\omega(\eps),\\
\int_{1-\omega(\eps)}^t\lambda_\eps(z)\, dz & 1-\omega(\eps)\le t\le 1+\omega(\eps),\\
t-1 & t\ge \omega(\eps)+1,
\end{cases}
\]
and
\[
\lambda_\eps(t)=\int_{-\infty}^{(t-1)/\omega(\eps)}\rho(z)\, dz.
\]
In particular $\Lambda_\eps(t)$ converges to
\[
\Lambda(t):=\begin{cases} 0, & 0\le t\le 1,\\
t-1, & t\ge 1,
\end{cases}
\]
for $\eps\to 0$. When $t\neq 1$ we are under the hypotheses of Proposition \ref{prop_micro_cla}. Indeed, $\nabla_{(t,x)}(x-\Lambda(t))\eta=(-\Lambda'(t)\eta,\eta)\neq 0$ for $\eta\neq 0$. It follows that
\[
\WF_{\Ginf} (u|_{0\le t<1})\subseteq\{(t,0,0,\eta):\, \eta\neq 0,\ 0\le t<1\}
\]
and
\[
\WF_{\Ginf} (u|_{t>1})\subseteq\{(t,t-1,-\eta,\eta):\, \eta\neq 0,\ t>1\}.
\]
The fact that $\Lambda$ is not differentiable at $t=1$ does not allow us to apply Proposition \ref{prop_micro_cla} in a neighborhood of $t=1$. The wavefront set at $t=1$ has been computed by direct calculations in \cite{O:07}.
\end{example}
Another class of generalised Fourier integral operators satisfying condition \eqref{rel_cond} can be found in solving a hyperbolic Cauchy problem where the principal part depends only on $t$. More precisely, we will deal with phase functions of the type
\[
[(x\eta+\varphi_\eps(t,\eta)-y\eta)_\eps],
\]
where $\varphi_\eps$ is homogeneous of order $1$ in $\eta$. Note that differently from Proposition \ref{prop_micro_cla} we do not require any convergence as $\eps\to 0$ but simply that $\varphi_\eps$ does not depend on $x$. Let us consider the Cauchy problem
\[
 D_t u =\sum_{j=1}^n a_{1,j}(t)D_j u+a_0(t,x)u,\qquad\quad u(0,\cdot)=u_0\in\Gc(\R^n).
\]
A typical example of $\varphi=[(\varphi_\eps)_\eps]$ is given by
\[
\varphi(t,\eta)=\sum_{j=1}^n\biggl(\int_0^t a_{1,j}(\sigma)\, d\sigma\biggr) \eta_j,
\]
where $\gamma_j(x,t,s)=x_j+\int_s^t a_{1,j}(\sigma)\, d\sigma$ are the components of the characteristic curves corresponding to the operator $D_t-\sum_{j=1}^n a_{1,j}(t)D_j$.
\begin{proposition}
\label{prop_micro_cla_2}
Let
\[
u(t,x)=\int_{\R^n}\esp^{i[(x\eta+\varphi_\eps(t,\eta))_\eps]}b(t,x,\eta)\widehat{u_0}(\eta)\, \dslash\eta:=Bu_0,
\]
where $b$ is a regular generalised symbol of order $0$ and $[(x\eta+\varphi_\eps(t,\eta)-y\eta)_\eps]$ is a slow scale generalised phase function in the variable $(y,\eta)$. If the net $(\sup_{|t|\le T,\eta\neq 0}\partial_t\varphi_\eps(t,\frac{\eta}{|\eta|}))_\eps$ is bounded for all $T>0$ then
\beq
\label{micro_cla_2}
\WF_{\Ginf}u\subseteq  \big(\WF'_{\Ginf}K_B\circ\WF_{\Ginf} u_0\big) \cup \big(\WF'_{\Ginf}K_B\circ(\singsupp_{\Ginf} u_0\times\{0\})\big).
\eeq
\end{proposition}
\begin{proof}
The microlocal result \eqref{micro_cla_2} is directly obtained from Theorem \ref{theo_ChazP} by checking that the inclusion
\[
\WF'_{\Ginf}K_B\circ \WF_{\Ginf}u_0\subseteq\CO{\R\times\R^n}
\]
holds. We will prove something stronger. We will prove that if $(t_0,x_0,y_0,\tau_0,\xi_0,\eta_0)\in\WF_{\Ginf}K_B$ then $\xi_0=-\eta_0$. Thus $(t_0,x_0,y_0,0,0,\eta_0)\not\in\WF_{\Ginf}K_B$ if $\eta_0\neq 0$.

By Theorem \ref{theo_GHO}, for the point $(t_0,x_0,y_0,\tau_0,\xi_0,\eta_0)\in\WF_{\Ginf}K_B$ we can find sequences of neighborhoods $U_n(t_0,x_0,y_0)$, $\Gamma_n(\tau_0,\xi_0,\eta_0)$ and sequences of points $(t_n,x_n,y_n)\in U_n(t_0,x_0,y_0)$, $(\tau_n,\xi_n,\eta_n)\in \Gamma_n$ and $\theta_n$ such that
\[
|(\tau_n,\xi_n,\eta_n)-(\partial_t\varphi_{\eps_n}(t_n,\theta_n),\theta_n,-\theta_n)|\le |\log \eps_n|^{-1}(|(\tau_n,\xi_n,\eta_n)|+|\theta_n|),
\]
where $\eps_n$ tends to $0$. Again we may assume that $\theta_n$ has norm $1$. Hence,
\[
|(\tau_n,\xi_n,\eta_n)-(\partial_t\varphi_{\eps_n}(t_n,\theta_n),\theta_n,-\theta_n)|\le |\log \eps_n|^{-1}(|(\tau_n,\xi_n,\eta_n)|+1).
\]
From the hypothesis on $\partial_t\varphi_\eps$ it follows that
\[
|(\tau_n,\xi_n,\eta_n)|\le|\log \eps_n|^{-1}(|(\tau_n,\xi_n,\eta_n)|+1)+|\partial_t\varphi_{\eps_n}(t_n,\theta_n),\theta_n,-\theta_n)| \le |\log \eps_n|^{-1}|(\tau_n,\xi_n,\eta_n)|+c,
\]
which means that the sequence $|(\tau_n,\xi_n,\eta_n)|$ is bounded. Passing to subsequences we have that $(\tau_n,\xi_n,\eta_n)$ converges to $\lambda(\tau_0,\xi_0,\eta_0)$ and $\theta_n$ to some $\theta_0\neq 0$. Finally,
\[
|(\xi_n,\eta_n)-(\theta_n,-\theta_n)|\le |(\tau_n,\xi_n,\eta_n)-(\partial_t\varphi_{\eps_n}(t_n,\theta_n),\theta_n,-\theta_n)|\le |\log \eps_n|^{-1}(|(\tau_n,\xi_n,\eta_n)|+1)
\]
yields
\[
(\lambda\xi_0,\lambda\eta_0)=(\theta_0,-\theta_0),
\]
i.e., $\xi_0=-\eta_0$.
\end{proof}
\subsection{The Hamiltonian flow and the $\Ginf$-wave front set at fixed time}
We now fix $t$ and investigate the microlocal properties of the solution $u(t,\cdot)$ of the Cauchy problem \eqref{first_order_equ}. As in the previous subsection we begin with some theoretical results on generalised Fourier integral operators of the type
\[
F_\phi(b)u(x)=\int_{\R^n}\esp^{i\phi(x,\eta)}b(x,\eta)\widehat{u}(\eta)\, \dslash\eta
\]
and we then consider the specific case of a strictly hyperbolic Cauchy problem. The following theorem is modelled on \cite[Theorem 4.1.6]{MR:97} and makes use of the composition formula for generalised operators in \cite{Garetto:ISAAC07}.
\begin{theorem}
\label{theo_MascaRo}
Let $\Om$ be an open subset of $\R^n$, $\phi$ a slow scale regular generalised symbol on $\Om_x\times\R^n_\eta$ homogeneous of order $1$ in $\eta$, $b$ a slow scale regular generalised symbol of order $0$ on $\Om\times\R^n$ compactly supported in $x$ and $u\in\Gc(\Om)$. Let $(x_0,\xi_0)\in\CO{\Om}$ and $(y_0,\eta_0)\not\in\WF_{\Ginf}u$. If there exist
\begin{itemize}
\item[-] a representative $(\phi_\eps)_\eps$ of $\phi$ such that $\nabla_x\phi_\eps\neq 0$ and
$$\big(\inf_{x\in K,\eta\in\R^n\setminus 0}\big|\nabla_x\phi_\eps(x,\frac{\eta}{|\eta|})\big|\big)_\eps$$
is slow scale strictly non-zero for all $K\Subset\Om$,
\item[-] conic neighborhoods $W$ and $Z$ of $(y_0,\eta_0)$ and $(x_0,\xi_0)$, respectively, with $W\subset(\WF_{\Ginf}u)^{\rm{c}}$,
\item[-] $\eps_0\in(0,1]$ and a slow scale net $(\lambda_\eps)_\eps$
\end{itemize}
such that the following implication holds
\beq
\label{impl_imp}
(x,\nabla_x\phi_\eps(x,\eta))\in Z\quad \Rightarrow \quad |y-\nabla_\eta\phi_\eps(x,\eta)|\ge \lambda^{-1}_\eps\quad \text{for all $\eps\in(0,\eps_0]$ and all $(y,\eta)\not\in W$},
\eeq
then
\beq
\label{impl_imp_1}
(x_0,\xi_0)\not\in\WF_{\Ginf}(F_\phi(b)u).
\eeq
\end{theorem}
\begin{proof}
By definition of the $\Ginf$-wavefront set we can find a conic neighborhood $W_1$ of $(y_0,\eta_0)$ such that $W\subset W_1\subset (\WF_{\Ginf}u)^{\rm{c}}$. From Lemma 3.4 in \cite{GH:05} we can find a symbol $c\in S^0(\Om\times\R^{n})$ with $\musupp\, c \subseteq W_1$ and such that $c(x,\eta)=1$ on $W$ when $|\eta|\ge 1$. In addition, we can construct $c$ so that $\supp_x c\Subset\Om$. From Theorem 3.6 in \cite{GH:05} we have
$\WF_{\Ginf} c(x,D)u\subseteq \WF_{\Ginf}u\cap \musupp\, c=\emptyset$. It follows that $c(x,D)u\in\Gcinf(\Om)$. We now write
\[
F_\phi(b)u = F_\phi(b)c(y,D)u+F_\phi(b)e(y,D)u,
\]
where $e(y,\eta)=1-c(y,\eta)$. Since $c(y,D)u\in\Gcinf(\Om)$ we obtain from the mapping properties of a generalised Fourier integral operator that $F_\phi(b)c(y,D)u$ is $\Ginf$-regular as well. This means
\[
\WF_{\Ginf}F_\phi(b)u =\WF_{\Ginf}F_\phi(b)e(y,D)u.
\]
We will therefore prove that the implication
\beq
\label{impl_imp_2}
(y_0,\eta_0)\not\in\WF_{\Ginf}(u)\qquad\Rightarrow\qquad (x_0,\xi_0)\not\in\WF_{\Ginf}(F_\phi(b)e(y,D)u)
\eeq
is true. Making use of the characterisation of the $\Ginf$-wave front set stated in \cite[Theorem 3.11]{GH:05} it is sufficient to prove that for some pseudodifferential operator $a(x,D)$ with $a\in S^0(\Om\times\R^n)$ such that $\supp\,a\subseteq Z$ and $a(x,\xi)=1$ on some smaller neighborhood $Z_1$ when $|\xi|\ge 1$, we have $a(x,D)F_\phi(b)(e(y,D)u)\in\Ginf(\Om)$. We are under the hypotheses of the composition theorem in \cite{Garetto:ISAAC07} (Theorem 5.11 in \cite{Garetto:ISAAC07} without dependence in $t$). Hence, modulo a $\Ginf$-regularising operator which does not affect the wavefront set, we have that the operator $a(x,D)F_\phi(b)(e(y,D)u)$ is the generalised Fourier integral operator
\[
F_\phi(b_1)((e(y,D)u))(x)=\int_{\R^n}\esp^{i\phi(x,\eta)}b_1(x,\eta)\int_{\Om}\esp^{-iy\eta}e(y,D)u(y)\, dy\, \dslash\eta,
\]
where $b_1\in\G^\ssc_{S^{0}(\Om\times\R^n)}$ has the following asymptotic expansion:
\[
\sum_\alpha\frac{\partial^\alpha_\xi a(x,\nabla_x\phi(x,\eta))}{\alpha!}D^\alpha_z\big(\esp^{i\overline{\phi}(z,x,\eta)}b(z,\eta)\big)|_{z=x}.
\]
At the level of representatives it means
\[
b_{1,\eps}(x,\eta)\sim\sum_\alpha\frac{\partial^\alpha_\xi a(x,\nabla_x\phi_\eps(x,\eta))}{\alpha!}D^\alpha_z\big(\esp^{i\overline{\phi_\eps}(z,x,\eta)}b_\eps(z,\eta)\big)|_{z=x}
\]
and since we argue modulo $\Ginf$ we can consider $b_{1,\eps}$ as a sum of a convergent series obtained from the asymptotic expansion above as in \cite[Theorem 2.2]{Garetto:ISAAC07}. Hence $b_{1,\eps}(x,\eta)=0$ if $(x,\nabla_x\phi_\eps(x,\eta))\not\in Z$ and $\eps\in(0,\eps_0]$. In the sequel we will complete the proof by showing that $F_\phi(b_1)e(y,D)$ is an integral operator with kernel in $\Ginf(\Om\times\Om)$. This will follow from the fact that the transposed operator ${\,}^t e(y,D){\,}^t F_{\phi}(b_1)$ has kernel in $\Ginf(\Om\times\Om)$. Working at the level of representatives we have that the kernel of ${\,}^t e(y,D){\,}^t F_{\phi}(b_1)$ is given by
\beq
\label{int_kern}
v\to\int_{\Om\times\Om\times\R^{n}}\esp^{i(\phi_\eps(x,\eta)-y\eta)}\big({\,}^t e\big)(y,-\eta)b_{1,\eps}(x,\eta)v(x,y)\, dx\, dy\, \dslash\eta,
\eeq
where ${\,}^t e(y,\eta)\sim\sum_{\alpha}(\alpha !)^{-1} D^\alpha_\eta D^\alpha_y e(y,-\eta)$. Again arguing modulo $\Ginf$ it is not a restriction to assume that ${\,}^t e(y,-\eta)=0$ when $(y,\eta)\in W$. It follows that in the integral \eqref{int_kern} we may assume that $(x,\nabla_x\phi_\eps(x,\eta))\in Z$ and $(y,\eta)\not\in W$ for $\eps\in(0,\eps_0]$. Thus, by integrating by parts we have
\begin{multline*}
\int_{\Om\times\Om\times\R^{n}}\esp^{i(\phi_\eps(x,\eta)-y\eta)}\big({\,}^t e\big)(y,-\eta)b_{1,\eps}(x,\eta)v(x,y)\, dx\, dy\, \dslash\eta\\
=\int_{\Om\times\Om\times\R^{n}}\esp^{i(\phi_\eps(x,\eta)-y\eta)}\frac{\Delta^N_\eta\big({\,}^t e(y,-\eta)b_{1,\eps}(x,\eta)\big)}{|\nabla_x\phi_\eps(x,\eta)-y|^{2N}}\dslash\eta\, v(x,y)\, dx\, dy,
\end{multline*}
where we can make use of the hypothesis \eqref{impl_imp}. Since the nets involved in the integral
\beq
\label{int_eta}
\int_{\R^n}\esp^{i(\phi_\eps(x,\eta)-y\eta)}\frac{\Delta^N_\eta\big({\,}^t e(y,-\eta)b_{1,\eps}(x,\eta)\big)}{|\nabla_x\phi_\eps(x,\eta)-y|^{2N}}\dslash\eta
\eeq
are of slow scale type and ${|\nabla_x\phi_\eps(x,\eta)-y|^{2N}}\ge \lambda_\eps^{-2N}$ where $(\lambda_\eps)_\eps$ is a slow scale net as well, we conclude that \eqref{int_eta} generates a generalised function in $\Ginf(\Om\times\Om)$.
\end{proof}
Theorem \ref{theo_MascaRo} applies to the solution
\[
u(t,x)=F_\phi(b)(u_0)(t,x)=\int_{\R^n}\esp^{i\phi(t,x,\eta)}b(t,x)\widehat{u_0}(\eta)\, \dslash\eta
\]
(see Proposition \ref{prop_FIO_formula}) of the Cauchy problem \eqref{first_order_equ} under a certain convergence assumption on the Hamiltonian flow.
\begin{theorem}
\label{theo_claudia_flow}
Let $u\in\G(\R^{n+1})$ be the unique solution of the Cauchy problem
\[
\begin{split}
 D_t u &=\sum_{j=1}^n a_{1,j}(t,x)D_j u+a_0(t,x)u, \\
 u(0,\cdot)&=u_0\in\Gc(\R^n).
\end{split}
\]
Under the hypotheses $(h2)$ on the coefficients $a_{1,j}$, $j=1,...,n$, and $(i2)$ on the coefficients $a_0$, let $\phi(t,x,\eta)=\sum_{h=1}^n\gamma_h(x,t,0)\eta_h$ be the solution of the eikonal equation \eqref{eikonal_0} and $b(t,x)\in\Ginf(\R^{n+1})$ as in \eqref{FIO_formula}. Let $\phi$ have a representative $(\phi_\eps)_\eps$ satisfying the following two conditions for any $t\in\R$:
\begin{itemize}
\item[(i)] $\nabla_x\phi_\eps(t,x,\eta)\neq 0$ for all $\eps,x$ and $\eta\neq 0$ and the net
\[
\inf_{x\in K\Subset\R^n, \eta\neq 0}|\nabla_x\phi_\eps\big(t,x,\frac{\eta}{|\eta|}\big)|
\]
is slow scale strictly non-zero;
\item[(ii)] the Jacobian $(\partial_j\gamma_{i,\eps}(x,t,0))_{i,j}$ is invertible and the solution $(y,\eta)=\chi_{t,\eps}(x,\xi)$ of the system
\begin{align*}
\xi&=\nabla_x\phi_\eps(t,x,\eta),\\
y&=\nabla_\eta\phi_\eps(t,x,\eta),
\end{align*}
converges to a limit $\chi_t$ in $\mathcal{C}(\CO{\R^n},\CO{\R^n})$ as $\eps\to0$;
\item[(iii)] $\chi_t$ defines a bijection on $\CO{\R^n}$.
\end{itemize}
Then
\[
\WF_{\Ginf}u(t,\cdot)\subseteq\chi_t^{-1}(\WF_{\Ginf}u_0).
\]
\end{theorem}
\begin{proof}
We denote the linear map corresponding to the Jacobian $(\partial_j\gamma_{i,\eps}(x,t,0))_{i,j}$ by $\Gamma_\eps(t,x)$. The invertibility of $\Gamma_\eps(t,x)$ allows us to solve the Hamilton-Jacobi system above with respect to the variable $(y,\eta)$. More precisely we have
\[
\chi_{t,\eps}(x,\xi)=(\gamma_{1,\eps}(x,t,0),...,\gamma_{n,\eps}(x,t,0), (\Gamma_\eps^{\ast})^{-1}(t,x)\xi)
\]
and
\[
\chi^{-1}_{t,\eps}(y,\eta)=(\gamma_{1,\eps}(y,0,t),...,\gamma_{n,\eps}(y,0,t); \Gamma_\eps^{\ast}(t,\gamma_\eps(y,0,t)) \eta),
\]
where $\gamma_\eps(y,0,t)=(\gamma_{1,\eps}(y,0,t),...,\gamma_{n,\eps}(y,0,t))$.

Let $(x_0,\xi_0)\in \CO{\R^n}$ and $(y_0,\eta_0)=\chi_t(x_0,\xi_0)$. The limit in $(ii)$ entails the following assertion: for all conic neighborhoods $W$ of $(y_0,\eta_0)$ there exists a conic neighborhood $Z$ of $(x_0,\xi_0)$ and some $\eps_0\in(0,1]$ such that $\chi_{t,\eps}(x,\xi)\in W$ for all $(x,\xi)\in Z$ and $\eps\in(0,\eps_0]$. In other words, if $(x,\nabla_x\phi_\eps(t,x,\eta))\in Z$ and $\eps\in(0,\eps_0]$, then $(\nabla_\eta\phi_\eps(t,x,\eta),\eta)\in W$. Thus, $(x,\nabla_x\phi_\eps(t,x,\eta))\in Z$ and $\eps\in(0,\eps_0]$ imply for all $W_1$ with $\overline{W}\subseteq W_1$ the assertion
\[
\exists c>0\, \forall (y,\eta)\not\in W_1\qquad\quad |y-\nabla_\eta\phi_\eps(t,x,\eta)|\ge c.
\]
Combining the hypothesis $(i)$ with this result we are under the assumptions of Theorem \ref{theo_MascaRo} for fixed $t$. Therefore, $(y_0,\eta_0)=\chi_t(x_0,\xi_0)\not\in\WF_{\Ginf}u_0$ implies $(x_0,\xi_0)\not\in\WF_{\Ginf}u(t,\cdot)$ or in other words
\beq
\label{micro_incl_1}
\WF_{\Ginf}(u(t,\cdot))\subseteq \chi_t^{-1}(\WF_{\Ginf}u_0).
\eeq
\end{proof}
\begin{remark}
\label{rem_GH05}
When the coefficients $a_{1,j}$ and $a_0$ are $\Cinf$-functions Theorem \ref{theo_claudia_flow} recovers the propagation of singularities result obtained in \cite[Proposition 4.3]{GH:05}. In this case arguing by time reversal one proves the inclusion $\chi_t^{-1}(\WF_{\Ginf}u_0)\subseteq\WF_{\Ginf}u(t,\cdot)$ and therefore $\WF_{\Ginf}u(t,\cdot)=\chi_t^{-1}(\WF_{\Ginf}u_0)$.
\end{remark}

\subsection{Examples}
This subsection contains some examples of first order hyperbolic Cauchy problems whose principal part depends only on $t$. Applying the results of the previous sections we can express the Colombeau solution by means of a generalised FIO formula and perform a microlocal investigation.
\begin{proposition}
\label{prop_princ_t}
Let $u\in\G(\R^{n+1})$ be the unique solution of the Cauchy problem
\[
\begin{split}
 D_t u &=\sum_{j=1}^n a_{1,j}(t)D_j u+a_0(t,x)u, \\
 u(0,\cdot)&=u_0\in\Gc(\R^n).
\end{split}
\]
Under the hypotheses $(h2)$ and $(i2)$ let $\phi(t,x,\eta)=\sum_{h=1}^n\gamma_h(x,t,0)\eta_h$ be the solution of the eikonal equation \eqref{eikonal_0} and let $b(t,x)\in\Ginf(\R^{n+1})$ be as in \eqref{FIO_formula}. If there exists a choice of representatives $(a_{1,j,\eps})_\eps$ such that $\Lambda_{j,\eps}(t)=-\int_0^t a_{1,j,\eps}(z)\, dz$, $j=1,...,n$ depends continuously on $\eps$ and
\[
\ \ \ \ \ \ \ \  \ \ \ \ \ \ \ \ \ \lim_{\eps\to 0}\Lambda_{j,\eps}(t)=\Lambda_j(t),\qquad\qquad\qquad\quad \text{for all $t\in\R$, $j=1,...,n$},
\]
then
\[
\WF_{\Ginf}u(t,\cdot)=\chi_t^{-1}(\WF_{\Ginf}u_0),
\]
where
\[
\chi_t^{-1}(y,\eta)=(y_1+\Lambda_1(t),...,y_n+\Lambda_n(t),\eta).
\]
\end{proposition}
\begin{proof}
The components $\gamma_{j,\eps}(x,t,s)$ of the characteristic curves are the solutions of the Cauchy problems
\[
\frac{d}{ds}\gamma_{j,\eps}(x,t,s)=-a_{1,j,\eps}(s),\qquad\qquad \gamma_{j,\eps}(x,t,t)=x_j.
\]
Hence,
\[
\gamma_{j,\eps}(x,t,s)=\Lambda_{j,\eps}(s)+x_j-\Lambda_{j,\eps}(t)
\]
and
\[
\phi_\eps(t,x,\eta)=\sum_{j=1}^n(x_j-\Lambda_{j,\eps}(t))\eta_j.
\]
It follows that the hypotheses $(i)$ and $(ii)$ of Theorem \ref{theo_claudia_flow} are satisfied with $\chi_{t,\eps}(x,\xi)$ given by
\begin{align*}
\xi&=\eta,\\
y&=\gamma_\eps(x,t,0)=x-\Lambda_\eps(t).
\end{align*}
Now
\[
\chi_{t}(x,\xi)=\lim_{\eps\to 0}\chi_{t,\eps}(x,\xi)=(x-\Lambda(t),\xi)
\]
is a bijection on $\CO{\R^n}$ with $\chi_{t}^{-1}(y,\eta)=(y+\Lambda(t),\eta)$. From Theorem \ref{theo_claudia_flow} we have
\[
\WF_{\Ginf}(u(t,\cdot))\subseteq\chi_{t}^{-1}(\WF_{\Ginf}u_0)=\{(y+\Lambda(t),\eta):\ (y,\eta)\in\WF_{\Ginf}u_0\}.
\]
We will now argue by time reversal. We fix $t_0\in\R$ and we set $v(t,x)=u(t_0-t,x)$. From the original Cauchy problem we obtain
\[
D_t v=-\sum_{j=1}^n a_{1,j}(t_0-t)D_j v(t,x)-a_0(t_0-t,x)v
\]
with initial condition $v(0,x)=u(t_0,x)$. The corresponding characteristic curves are
\[
\gamma_{j,\eps}(x,t,s)=x_j+\int_{0}^s a_{j,\eps}(t_0-z)\, dz -\int_{0}^t a_{j,\eps}(t_0-z)\, dz
\]
and therefore
\[
\gamma_{j,\eps}(x,t,0)=x_j-\Lambda_{j,\eps}(t_0-t)+\Lambda_{j,\eps}(t_0).
\]
In particular, denoting the corresponding Hamiltonian flow by $\wt{\chi}_{t_0,\eps}^{-1}$, we have
\[
\wt{\chi}_{t_0,\eps}(x,\xi)=(x+\Lambda_\eps(t_0),\xi),\qquad \wt{\chi}_{t_0,\eps}^{-1}(y,\eta)=(y-\Lambda_\eps(t_0),\eta).
\]
Moreover
\[
\wt{\chi}_{t_0}(x,\xi):=\lim_{\eps\to 0}\wt{\chi}_{t_0,\eps}(x,\xi)=\lim_{\eps\to 0}(x+\Lambda_\eps(t_0),\xi)=(x+\Lambda(t_0),\xi)
\]
and
\[
\wt{\chi}_{t_0}^{-1}(y,\eta):=\lim_{\eps\to 0}\wt{\chi}_{t_0,\eps}^{-1}(y,\eta)=(y-\Lambda(t_0),\eta).
\]
Hence,
\[
\WF_{\Ginf}(v(t_0,\cdot))\subseteq \wt{\chi}_{t_0}^{-1}(\WF_{\Ginf}(v(0,\cdot)))
\]
or in other words
\[
\WF_{\Ginf}(u_0)\subseteq\wt{\chi}_{t_0}^{-1}(\WF_{\Ginf}(u(t_0,\cdot)))
\]
which implies
\[
\wt{\chi}_{t_0}(\WF_{\Ginf}(u_0))\subseteq \WF_{\Ginf}(u(t_0,\cdot)).
\]
Since $\wt{\chi}_{t_0}=\chi_{t_0}^{-1}$ we deduce
\[
\chi_{t_0}^{-1}(\WF_{\Ginf}u_0)\subseteq \WF_{\Ginf}(u(t_0,\cdot)).
\]
In conclusion,
\[
\WF_{\Ginf}u(t,\cdot)=\chi_t^{-1}(\WF_{\Ginf}u_0).
\]
\end{proof}
\begin{example}
\label{ex_MO}
{\bf A transport equation with discontinuous coefficient and distributional data.}\\
We continue studying the Cauchy problem from Example~\ref{rem_ex_MO}. By applying the previous proposition we can undertake a microlocal investigation of the solution $u\in\G([0,+\infty)\times\R)$. Using the nets and notations of Example~\ref{rem_ex_MO} we easily see that the assumptions of Proposition \ref{prop_princ_t} are fulfilled. In particular, the Hamiltonian flow is given by
\[
\chi_{t,\eps}(x,\xi)=(x-\Lambda_\eps(t),\xi),\qquad\qquad \chi_{t,\eps}^{-1}(y,\eta)=(y+\Lambda_\eps(t),\eta),
\]
with $\Lambda_\eps(t)$ defined in Example~\ref{rem_ex_MO} and having limit
\[
\Lambda(t)=\begin{cases} 0, & 0\le t\le 1,\\
t-1, & t\ge 1,\\
\end{cases}
\]
Hence
\[
\chi_t(x,\xi):=\lim_{\eps\to 0}\chi_{t,\eps}(x,\xi)=(x-\Lambda(t),\xi),\qquad\qquad \chi_t^{-1}(y,\eta)=\lim_{\eps\to 0}\chi_{t,\eps}^{-1}(y,\eta)=(y+\Lambda(t),\eta).
\]
Concluding, for each $t\ge 0$ we have
\[
\WF_{\Ginf}(u(t,\cdot))=\chi_t^{-1}(\WF_{\Ginf}\iota(\delta))=\{(\Lambda(t),\eta):\ \eta\neq 0\}.
\]
\end{example}
\begin{example}
{\bf An example in higher dimension.}\\
We can adapt the previous example to higher space dimensions. For instance, let
\begin{align*}
D_t u(t,x_1,x_2)&=-a_1H(t-1)D_{x_1} u(t,x_1,x_2)-a_2H(t-1)D_{x_2}u(t,x_1,x_2),\\
u(0,x)&=\psi(x_1,x_2)\delta_C(x_1,x_2),
\end{align*}
where $a_1,a_2\in\R$, $\delta_C$ is the Dirac measure along a smooth, simple curve $C\in\R^2$, and $\psi\in\D(\R^2)$. In this case the components of the characteristic curves are given by
\[
\gamma_{1,\eps}(x_1,t,s)=a_1\Lambda_\eps(s)+x_1-a_1\Lambda_\eps(t),\qquad \gamma_{2,\eps}(x_2,t,s)=a_2\Lambda_\eps(s)+x_2-a_2\Lambda_\eps(t)
\]
with $\Lambda_\varepsilon$ as above. This yields
\[
\chi_{t,\eps}(x,\xi)=((x_1-a_1\Lambda_\eps(t),x_2-a_2\Lambda_\eps(t)),(\xi_1,\xi_2)),\qquad \chi_{t,\eps}^{-1}(y,\eta)=((y_1+a_1\Lambda_\eps(t),y_2+a_2\Lambda_\eps(t)),(\eta_1,\eta_2))
\]
and the limit
\[
\chi_t^{-1}(y,\eta)=((y_1+a_1\Lambda(t),y_2+a_2\Lambda(t)),(\eta_1,\eta_2)).
\]
In conclusion,
\begin{align*}
\WF_{\Ginf}(u(t,\cdot))&=\chi_t^{-1}(\WF_{\Ginf}(\psi\delta_C))\\
&=\{(x_1+a_1\Lambda(t),x_2+a_2\Lambda(t),\xi_1,\xi_2):\, (x_1,x_2)\in\supp\psi,\ (\xi_1,\xi_2)\in N(x_1,x_2)\}
\end{align*}
where $N(x_1,x_2)$ denotes the set of nonzero conormal directions of $C$.
\end{example}
\begin{example}
{\bf An example with the Dirac measure as a coefficient.}\\
We consider the Cauchy problem
\begin{align*}
D_t u(t,x)&=-\delta(t-1)D_x u(t,x),\\
u(0,x)&=u_0(x),
\end{align*}
with $u_0\in\Gc(\R)$. In the regularisation of $\delta(t-1)$ via convolution with a mollifier we use $\rho\in\D(\R)$ as in Example~\ref{rem_ex_MO} and a slow scale net $(\omega^{-1}(\eps))_\eps$ with $0<\omega(\eps)<1$ and $\lim_{\eps\to 0}\omega(\eps)=0$.
The solution is represented by
\[
   u_\varepsilon(t,x) = u_{0\varepsilon}(x-\lambda_\varepsilon(t))
\]
with $\lambda_\varepsilon(t)$ as in Example~\ref{rem_ex_MO}.
Since  $\supp\rho\subseteq[-1,1]$ and $\int_{-1}^0\rho(z)\, dz=1/2$ we have that $\lambda_\varepsilon(1) = \frac12$ for all $\eps\in(0,1]$ and thus
\[
\lim_{\eps\to 0}\lambda_{\eps}(t)=:\lambda(t)=\begin{cases} 1, & t>1,\\
\frac{1}{2}, & t=1,\\
0, & t< 1.\\
\end{cases}
\]
By Proposition \ref{prop_princ_t},
\[
\WF_{\Ginf}(u(t,\cdot))=\{(x+\lambda(t),\xi):\ (x,\xi)\in\WF_{\Ginf}u_0\}
\]
for all $t\in\R$. Note that the limits of the flows $\chi_{t,\varepsilon}(x,\xi)$ are homeomorphisms at each fixed $t$, while the limiting two-dimensional characteristic coordinate change $(t,x) \to (t,x-\lambda(t))$ is still bijective, but no longer continuous. Such a situation is admitted by Proposition \ref{prop_princ_t}.
\end{example}

\bibliographystyle{abbrv}
\newcommand{\SortNoop}[1]{}

\end{document}